\documentclass[11pt]{article}
\usepackage{color}
\usepackage{amssymb, graphicx, amsmath,epstopdf}
\usepackage{amsthm}
\usepackage{algorithm}
\usepackage{algpseudocode}
\usepackage{cases}
\usepackage{amsfonts,amsmath,amsthm,amsfonts,amssymb}
\usepackage{graphics,graphicx,subfigure}

\allowdisplaybreaks[4]

\newtheorem{theorem}{Theorem}[section]
\newtheorem{lemma}[theorem]{Lemma}

\newtheorem{remark}[theorem]{Remark}

\title{\textbf{\Large Implementation of the ADMM to Parabolic Optimal Control Problems with Control Constraints and Beyond}}
\begin{document}
\author{Yongcun Song\thanks{{{Department of Mathematics, The University of Hong Kong, Hong Kong, China. Email: ysong307$@$hku.hk.}}}
\quad Xiaoming Yuan\thanks{{{Department of Mathematics, The University of Hong Kong, Hong Kong, China. Email: xmyuan$@$hku.hk. This author was supported by the seed fund for basic research at The University of Hong Kong (project code: 201807159005) and a General Research Fund from Hong Kong Research Grants Council.}}}
\quad Hangrui Yue\thanks{{{Department of Mathematics, The University of Hong Kong, Hong Kong, China. Email: yuehangrui$@$gmail.com.}}}}
	\maketitle
	\vspace{-1.0cm}
	\begin{abstract}
\noindent\hrulefill\vskip2pt
\noindent Optimal control problems subject to both parabolic partial differential equation (PDE) constraints and additional constraints on the control variables are generally challenging, from either theoretical analysis or algorithmic design perspectives. Conceptually, the well-known alternating direction method of multipliers (ADMM) can be directly applied to such a problem. An attractive advantage of this direct ADMM application is that the additional constraint on the control variable can be untied from the parabolic PDE constraint; these two inherently different constraints thus can be treated individually in iterations. At each iteration of the ADMM, the main computation is for solving an optimal control problem with a parabolic PDE constraint while it is not interacted with the constraint on the control variable. Because of its inevitably high dimensionality after the space-time discretization, the parabolic optimal control problem at each iteration can be solved only inexactly by implementing certain numerical scheme internally and thus a two-layer nested iterative scheme is required. It then becomes important to find an easily implementable and efficient inexactness criterion to execute the internal iterations, and to prove the overall convergence rigorously for the resulting two-layer nested iterative scheme. To implement the ADMM efficiently, we propose an inexactness criterion that is independent of the mesh size of the involved discretization, and it can be executed automatically with no need to set empirically perceived constant accuracy a prior. The inexactness criterion turns out to allow us to solve the resulting optimal control problems with the only parabolic PDE constraints to medium or even low accuracy and thus saves computation significantly, yet convergence of the overall two-layer nested iterative scheme can be still guaranteed rigorously. Efficiency of this ADMM implementation is promisingly validated by preliminary numerical results. Our methodology can also be extended to a range of optimal control problems constrained by other linear PDEs such as elliptic equations, hyperbolic equations, convection-diffusion equations and fractional parabolic equations.

\bigskip

\rule{14.9cm}{0.1mm}
\end{abstract}

\section{Introduction}
Optimal control problems constrained by partial differential equations (PDEs) with additional constraints on the control and/or state variables capture important models in various areas, such as physics, chemistry, engineering, medicine and financial engineering. We refer to, e.g. \cite{glowinski1994exact, glowinski1995exact, glowinski2008exact, hinze2008optimization, lions1971optimal, troltzsch2010optimal}, for a few references. These problems are generally difficult from either theoretical analysis or algorithmic design perspectives; one reason is that the PDE constraints and other constraints on the control and/or state variables are coupled. The high dimensionality of the resulting algebraic systems after discretization further explains the lack of a rich set of efficient numerical schemes in the literature, especially for some optimal control problems with time-dependent PDE constraints. To tackle such a problem numerically, a general principle is that the structures and properties of the model should be sophisticatedly considered in algorithmic design, rather than applying some existing algorithms generically. One particular desire is to untie the PDE constraints (usually more difficult) and other constraints (usually much easier) on the control and/or state variables so that these two inherently different constraints can be treated individually in iterations.

\subsection{Model}
In this paper, we consider the following optimal control problem with a parabolic PDE constraint and a box constraint on the control variable:
\begin{equation}\label{Basic_Problem}
\begin{aligned}
\min_{u\in {\mathcal{C}}, y\in L^2(Q)} \quad  &\frac{1}{2}\iint_Q |y-y_d|^2 dxdt+\frac{\alpha}{2}\iint_{\mathcal{O}}|u|^2 dxdt\\
\end{aligned}
\end{equation}
subject to the state equation
\begin{equation}\label{state_eqn}
\left\{
\begin{aligned}
&\frac{\partial y}{\partial t}-\nu \Delta y+a_0y=u\chi_{\mathcal{O}},&\quad \text{in}\quad \Omega\times(0,T), \\
&y=0,&\quad \text{on}\quad \Gamma\times(0,T),\\
& y(0)=\varphi,&
\end{aligned}
\right.
\end{equation}
where $\Omega$ is an open bounded domain in $ \mathbb{R}^d$ $(d\geq 1)$ and $\Gamma=\partial\Omega$ is the piecewise continuous boundary of $\Omega$; $\omega$ is an open subset of $\Omega$ and $0<T<+\infty$; the domain $Q=\Omega\times (0, T)$ and $\mathcal{O}=\omega\times (0, T)$. In (\ref{Basic_Problem})--(\ref{state_eqn}), $u$ and $y$ are called the control variable and state variable, respectively. The target function $y_d$ is given in $L^2(Q)$ and the admissible set ${\mathcal{C}}$ is defined by
\begin{equation*}
{\mathcal{C}}=\{v|v\in L^\infty(\mathcal{O}), a\leq v(x; t)\leq b ~\text{a.e. in} \, \mathcal{O}\}\subset{L^2(\mathcal{O})}.
\end{equation*}
In addition, we denote by $\Delta :=\nabla\cdot\nabla $ the Laplace operator and $\chi_\mathcal{O}$ the characteristic function of the set $\mathcal{O}$. The constant $\alpha>0$ is a regularization parameter; $a$ and $b$ are given constants; the initial value $\varphi$ is given in $L^2(\Omega)$. The coefficients $a_0$   $(\geq0)\in L^{\infty}(Q)$ and $\nu$ is a positive constant. The problem (\ref{Basic_Problem})--(\ref{state_eqn}) has a wide range of applications in the areas of physics, chemistry and engineering, see, e.g., \cite{glowinski2008exact,troltzsch2010optimal}. Existence and uniqueness of the solution to the problem (\ref{Basic_Problem})--(\ref{state_eqn}) can be proved in a standard argument as studied in \cite{lions1971optimal}; we refer to \cite{troltzsch2010optimal} for the detail.

\subsection{Parabolic Optimal Control Problems without Control Constraints}

For the special case of the problem (\ref{Basic_Problem})--(\ref{state_eqn}) where $\mathcal{C}=L^2(\mathcal{O})$, i.e., there is no constraint on the control variable, the resulting problem is called an unconstrained parabolic optimal control problem and it has been well studied in some earlier literatures such as \cite{lions1971optimal} and some more recent ones such as \cite{troltzsch2010optimal}. There is a rich set of papers discussing how to solve unconstrained parabolic optimal control problems numerically; and methods in the literature can be generally categorized as the ``black-box'' and ``all-at-once'' approaches. The ``black-box'' approach commonly suggests substituting the state equation into the objective functional to eliminate the state variable $y$, and treats an unconstrained parabolic optimal control problem as an optimization problem with respect to the control variable $u$. Note that each iteration of a ``black-box'' approach requires solving the involved state equation. We refer to \cite{glowinski1994exact, glowinski2008exact} for some efficient ``black-box'' type numerical schemes for unconstrained parabolic control problems with different types of control variables. {On the other hand, the ``all-at-once'' approach keeps the state equation in the constraints, and treats both the state and control variables separately. The optimality condition of such a resulting constrained optimization problem after discretization can be represented as a linear saddle point system, which can be solved by some efficient iterative solvers such as Krylov subspace methods.} We refer to \cite{mcdonald2016all, pearson2012regularization, ulbrich2007generalized} for more details. Both ``black-box'' and ``all-at-once'' approaches can be combined with standard techniques such as domain decomposition methods and multi-grid methods to further improve their numerical performance; see, e.g., \cite{ barker2015domain, borzi2003multigrid, gander2016schwarz, heinkenschloss2005time,mathew2010analysis}, for some intensive study.

\subsection{SSN Methods for Parabolic Optimal Control Problems with Control Constraints}

In the literature, semi-smooth Newton (SSN) methods are state-of-the-art for various optimal control problems with control constraints.  For instance, SSN methods have been intensively studied for  optimal control problems with elliptic PDE constraints; see, e.g., \cite{hintermuller2002primal, hinze2008optimization, ulbrich2011semismooth} and reference therein. A common feature of SSN methods is that a semismooth Newton direction is constructed by using a generalized Jacobian in sense of Clarke (see \cite{clark1990}) and then a Newton iteration is expressed in terms of certain active set strategy which identifies the active and inactive indices iteratively in accordance with the control constraints, see, e.g., \cite{BKI99,porcelli2015preconditioning}. {In \cite{BKI99}, some adaptive strategies have been proposed to alleviate the computational load of the Newton iterations with the resulting iteratively varying coefficient matrices.}
As analyzed in \cite{hintermuller2002primal}, a SSN method with an active set strategy can be explained as the primal-dual active set (PDAS) strategy studied in \cite{BKI99} for certain problems such as linear-quadratic optimal control problems with box control constraints, including the problem (\ref{Basic_Problem})--(\ref{state_eqn}). The convergence of the PDAS approach can be found in \cite{KR2002} while some numerical results are also reported therein for parabolic boundary control problems with $d=1$. In \cite{hintermuller2002primal}, it has been proved that SSN methods possess locally superlinear convergence and usually can find high-precision solutions, on the condition that that some initial values can be deliberately chosen. Note that it is assumed by default that the resulting Newton systems should all be solved exactly to validate the theoretical analysis and hence the mentioned nice properties of SSN type methods. Computationally, it is notable that the Newton systems arising in SSN methods are usually ill-conditioned, and as commented in \cite{SW12} that ``it is never solved without the application of a preconditioner''. Seeking appropriate preconditioners so as to improve the spectral properties of the Newton systems is indeed a major factor to ensure the success of implementing a SSN type method. In the literature, e.g., \cite{HKV09,porcelli2015preconditioning,schiela2014operator,stoll2013one,ulbrich2015preconditioners}, some preconditioned iterative solvers were proposed for various SSN methods.

One motivation of considering SSN type methods for the general case of the problem (\ref{Basic_Problem})--(\ref{state_eqn}) with $\mathcal{C}\subsetneq L^2(\mathcal{O})$ is that the indicator function of the additional constraint on the control variable $u\in{\mathcal{C}}$ arising in the optimality condition of the problem (\ref{Basic_Problem})--(\ref{state_eqn}) is nonsmooth; hence gradient type methods are not applicable, see e.g., \cite{hinze2008optimization,troltzsch2010optimal}.
 But, a particular obstacle of applying SSN type methods to the problem (\ref{Basic_Problem})--(\ref{state_eqn}) is that the simple box constraint on the control variable is forced to be considered together with the main parabolic PDE (\ref{state_eqn}) simultaneously. Despite that the computational load of assembling the Newton systems can be alleviated by the adaptive strategies in \cite{BKI99}, the varying active sets require adjusting the preconditioners iteratively. Indeed, as commented in \cite{stoll2013one}, ``we have recomputed the preconditioner for every application involving a different active set'' and that ``the recomputation of the preconditioner needs to be avoided". {Hence, the simple constraint on the control variable unnecessarily complicates the Newton systems because of the request of active-set-dependent preconditioning, and this feature makes it difficult to apply SSN type methods to the problem (\ref{Basic_Problem})--(\ref{state_eqn}).}

Implementation of SSN type methods to the general case of the problem (\ref{Basic_Problem})--(\ref{state_eqn}) with $d\geq 2$ is further restrained by the inevitably high dimensionality of the resulting Newton systems. To elaborate, if we set the mesh sizes of both the time and space discretizations as $1/100$, then the dimensionality of the resulting Newton system at each iteration is order of $O(10^6)$ for $d=2$ and $O(10^8)$ for $d=3$. Hence, for some time-dependent problems such as (\ref{Basic_Problem})--(\ref{state_eqn}) with $d\geq 2$, it is not practical to solve such large-scale Newton systems either exactly or up to high precisions. It is thus necessary to discern some criterion that can be implemented easily, and to investigate the convergence if these Newton systems can only be solved up to certain levels of accuracy due to the difficulty of high dimensionality. In the literature, usually some empirically perceived constant accuracy is set a prior, and certainly fixing a constant accuracy by liberty may unnecessarily result in either too accurate computation (hence slower convergence) or too loose approximation (hence possible divergence) for the internal iterations\footnote{The same concerns also apply to the interior point methods in, e.g., \cite{PG2017}, for different types of optimal control problems.}. There seems still to lack of discussions on how to specify the inexactness criterion methodologically and how to prove the convergence of the resulting inexact executions rigorously in the literature of SSN methods. Also, as mentioned in, e.g., \cite{porcelli2015preconditioning}, some SSN methods require the accuracy for internal iterations to be increased when the mesh size for discretization becomes smaller. This essentially increases the computational load for solving the Newton systems and may significantly slow down the overall convergence if fine meshes are used to discretize the problem (\ref{Basic_Problem})--(\ref{state_eqn}).

\subsection{Conceptual Application of ADMM}

Inspired by the aforementioned difficulties in the consideration of implementing the well-studied SSN methods to the problem (\ref{Basic_Problem})--(\ref{state_eqn}), our first motivation is to design an algorithm that can treat the parabolic PDE constraint (difficult one) and the box constraint on the control variable (easy one) separately in its execution. A particular goal is that the subproblems associated with the parabolic PDE constraint arising in different iterations should have invariant coefficient matrices so that certain numerical strategy such as preconditioning can be uniformly applied. To this end, it suffices to consider the well-studied alternating direction method of multipliers (ADMM) which was first introduced by Glowinski and Marroco in \cite{glowinski1975approximation} for nonlinear elliptic problems.


Let us see how the ADMM can be applied to the problem (\ref{Basic_Problem})--(\ref{state_eqn}) and a prototype algorithm can be obtained immediately.  For this purpose, we let $S: L^2(\mathcal{O})\longrightarrow L^2(Q)$ be an affine solution operator associated with the state equation (\ref{state_eqn}); and it is defined as
\begin{equation}\label{def:S}
S(u):=y.
\end{equation}
It is clear that $S$ is bounded, continuous and compact. More properties of the operator $S$ can be referred to \cite{troltzsch2010optimal}. With $y=S(u)$, the problem (\ref{Basic_Problem})--(\ref{state_eqn}) can be rewritten as
\begin{equation*}
\min_{u\in {\mathcal{C}}} \quad  \frac{1}{2\alpha}\iint_Q |S(u)-y_d|^2 dxdt+\frac{1}{2}\iint_{\mathcal{O}}|u|^2 dxdt,
\end{equation*}
{which is actually a scaled version of the problem (\ref{Basic_Problem})--(\ref{state_eqn}).}
Further, by introducing an auxiliary variable $z\in L^2(\mathcal{O})$ such that $u=z$, the problem (\ref{Basic_Problem})--(\ref{state_eqn}) can be written as the following separable convex optimization problem
\begin{eqnarray}\label{equi_model}
\qquad\left\{
\begin{array}{lll}
\underset{(u, z)\in L^2(\mathcal{O})\times L^2(\mathcal{O})}{\min}\quad \tilde{J}(u)+I_{\mathcal{C}}(z) \\
\qquad\qquad\text{s.t.}\quad
\qquad u=z,\\
\end{array}
\right.
\end{eqnarray}
where $I_{\mathcal{C}}(\cdot)$ is the indicator function of the admissible set ${\mathcal{C}}$ and
\begin{equation}\label{reduced_objective}
\tilde{J}(u):=\frac{\gamma}{2}\iint_{Q}|S(u)-y_d|^2dxdt+\frac{1}{2}\iint_{\mathcal{O}}|u|^2dxdt,~\text{with}~\gamma=\frac{1}{\alpha}.
\end{equation}
The augmented Lagrangian functional associated with the problem (\ref{equi_model}) can be defined as
\begin{equation*}\label{augmented_Lagrangian}
L_\beta(u, z,\lambda):=\tilde{J}(u)+I_{\mathcal{C}}(z)-(\lambda,u-z)+\frac{\beta}{2}\|u-z\|^2,
\end{equation*}
in which $(\cdot,\cdot )$ and $\|\cdot\|$ are the canonical inner product and norm in $L^2(\mathcal{O})$, respectively; $\lambda\in L^2(\mathcal{O})$ is the Lagrange multiplier associated the constraint $u=z$, and $\beta>0$ is a penalty parameter.  To simplify the discussion, { the penalty parameter} is fixed throughout our discussion. Then, implementing the ADMM in \cite{glowinski1975approximation} to (\ref{equi_model}), we immediately obtain the scheme
\begin{subequations}\label{ADMM}
	\begin{numcases}
	~u^{k+1}=\arg\min_{u\in L^2(\mathcal{O})}L_\beta(u, z^{k},\lambda^k),\label{ADMM_u}\\
	z^{k+1} =\arg\min_{z\in L^2(\mathcal{O})}L_\beta(u^{k+1},z,\lambda^k),\label{ADMM_z}\\
	\lambda^{k+1} = \lambda^k-\beta(u^{k+1}-z^{k+1}).\label{ADMM_lambda}
	\end{numcases}
\end{subequations}


\subsection{Remarks on the Direct Application of ADMM}

The ADMM can be regarded as a splitting version of the classic augmented Lagrangian method (ALM) proposed in \cite{HMR69,PJD69}. At each iteration of the ALM, the subproblem is decomposed into two parts and they are solved in the Gauss-Seidel manner. A key feature of the ADMM is that the decomposed subproblems usually are much easier than the ALM subproblems and it becomes more likely to take advantage of the properties and structures of the model under investigation. Also, it generally does not require specific initial iterates to guarantee its satisfactory numerical performance. All these advantages make the ADMM a benchmark algorithm in various areas such as image processing, statistical learning, and so on; we refer to \cite{boyd2011distributed, glowinski2014alternating} for some review papers on the ADMM. In particular, the ADMM and its variants have been applied to solve some optimal control problems constrained by time-independent PDEs in, e.g.,\cite{attouch2008augmented, haoalternating, song2018two}. In \cite{GSY2019}, the ADMM was applied to parabolic optimal control problems with state constraints, and its convergence is proved without any assumption on the existence and regularity of the Lagrange multiplier. In \cite{glowinski2008exact}, the Peaceman--Rachford splitting method (see \cite{pr1955}) which is closely related to the ADMM was suggested to solve approximate controllability problems of parabolic equations numerically.

On the other hand, the ADMM is a first-order algorithm; hence its convergence is at most linear and it may not be efficient for finding very high-precision solutions. For a numerical scheme solving the problem (\ref{Basic_Problem})--(\ref{state_eqn}), total errors consist of the discretization error resulted by discretizing the model and the iteration error resulted by solving the discretized model numerically. In general, first-order numerical schemes such as the backward Euler finite difference method or piecewise constant finite element method with the step size $\tau$ is implemented for the time discretization (see e.g., \cite{glowinski2008exact,meidner2008priori}). As a result, the error order of the time discretization is $O(\tau)$ (see e.g., \cite{meidner2008priori}) and this estimate may dominate the magnitude of the total error. For such cases, pursuing too high-precision solutions of the discretized model does not help reduce the total error and it is more appropriate to just apply a first-order algorithm to find a medium-precision solution of the discretized model. This also motivates us to consider the ADMM (\ref{ADMM}) for the problem (\ref{Basic_Problem})--(\ref{state_eqn}).

\subsection{Difficulties and Goals}

It is straightforward to obtain the ADMM (\ref{ADMM}) for the problem (\ref{Basic_Problem})--(\ref{state_eqn}). But the scheme (\ref{ADMM}) is only conceptual, and it cannot be used immediately. As will be shown in Section \ref{sec:InexactADMM}, the $z$-subproblem (\ref{ADMM_z}) is easy; its closed-form solution can be computed by the projection onto the admissible set $\mathcal{C}$. But the $u$-subproblem (\ref{ADMM_u}) is essentially a standard unconstrained parabolic optimal control problem, and it can only be solved iteratively by certain existing algorithms. For instance, as studied in \cite{glowinski1994exact,glowinski2008exact}, we can apply the conjugate gradient (CG) method to solve it. Clearly, solving (\ref{ADMM_u}) dominates the computation of each iteration of the ADMM (\ref{ADMM}). Notice that the dimensionality of the time-dependent $u$-subproblem (\ref{ADMM_u}) after space-time discretization is inevitably high. Hence it is impractical to solve these subproblems too accurately. Meanwhile, there is indeed no necessity to pursue too accurate solutions for these subproblems, especially when the iterates are still far away from the solution point. Therefore, the subproblem (\ref{ADMM_u}) should be solved iteratively and inexactly, and the implementation of the ADMM (\ref{ADMM}) must be embedded by an internal iterative process for the subproblem (\ref{ADMM_u}). Interesting mathematical problems arise soon: How to determine an appropriate inexactness criterion to execute the internal iterations for solving the subproblem (\ref{ADMM_u}); and how to rigorously prove the convergence for the ADMM scheme (\ref{ADMM}) with two-layer nested iterations?

Preferably, the inexactness criterion for solving the subproblem (\ref{ADMM_u}) should be easy to implement, free of setting empirically perceived constant accuracy a prior, independent of space-time discretization mesh sizes and the regularization parameter $\alpha$, accurate enough to yield good approximate solutions which are good enough to ensure the overall convergence, yet efficient to avoid unnecessarily too accurate solutions so as to save overall computation. Moreover, though the convergence of the original ADMM has been well studied in both earlier literatures \cite{fortin1983augmented, gabay1975dual, glowinski2008lectures, glowinski1989augmented} and recent literatures \cite{he20121,he2015non}, the scheme (\ref{ADMM}) with the nested internal iterations subject to a given inexactness criterion should be analyzed from scratch. In short, our goals are: (I) proposing an easily implementable and appropriately accurate inexactness criterion for solving the subproblem (\ref{ADMM_u}) inexactly and hence an inexact version of the ADMM (\ref{ADMM}), (II) establishing the convergence for the resulting inexact ADMM rigorously, (III) specifying the inexact ADMM as concrete algorithms that are applicable to the problem (\ref{Basic_Problem})--(\ref{state_eqn}), and (IV) extending the inexact ADMM to other versions that can be used for a range of other optimal control problems.

\subsection{Organization}

The rest of this paper is organized as follows. In Section \ref{sec:InexactADMM}, we propose an inexactness criterion for the subproblem (\ref{ADMM_u}) and hence an inexact version of the ADMM for the problem (\ref{Basic_Problem})--(\ref{state_eqn}). Its strong global convergence is proved in Section \ref{sec:convergence}. In Section \ref{sec:convergence_rate}, its worst-case convergence rate measured by iteration complexity is established in both the ergodic and non-ergodic senses. We illustrate how to execute the new inexactness criterion and specify the inexact ADMM with implementation details in Section \ref{sec:implementation}. In Section \ref{sec:numerical}, some numerical results are reported to validate the efficiency of the proposed approach. {In Section \ref{se:extension}, we briefly discuss how to extend our analysis to other cases, including optimal control problems constrained by the wave equation with control constraints, and elliptic optimal control problems with control constraints.} Finally, some conclusions are made in Section \ref{sec:conclusion}.

\section{An Inexact ADMM}\label{sec:InexactADMM}
In this section, we first take a closer look at the solutions of the subproblems (\ref{ADMM_u})--(\ref{ADMM_lambda}), and then propose an inexactness criterion for solving the subproblem (\ref{ADMM_u}) iteratively. An inexact version of the ADMM (\ref{ADMM}) with two-layer nested iterations is thus proposed. For the simplicity of notations, hereinafter, we denote by $U$ and $Y$ the space $L^2(\mathcal{O})$ and $L^2(Q)$, respectively.

\subsection{Elaboration of Subproblems}

\subsubsection{Subproblem (\ref{ADMM_u})}

For the $u$-subproblem (\ref{ADMM_u}), it follows from
\begin{eqnarray*}
	L_\beta(u, z^{k}, \lambda^k)=\tilde{J}(u)
	-(\lambda^k, u-z^{k})+\frac{\beta}{2}\|u-z^{k}\|^2,
\end{eqnarray*}
that the $u$-subproblem (\ref{ADMM_u}) is equivalent to the following unconstrained parabolic optimal control problem:
\begin{equation*}
\underset{u\in U}{\min} \; j_k(u):=\tilde{J}(u)-(\lambda^k,u-z^{k})+\frac{\beta}{2}\|u-z^{k}\|^2.
\end{equation*}
Let $Dj_k(u)$ be the first-order derivative of $j_k$ at $u$. By perturbation analysis discussed in \cite{glowinski1994exact, glowinski2008exact}, we have
\begin{equation*}\label{gradient}
Dj_k(u)=u+ p|_{\mathcal{O}}+\beta(u-z^{k})-\lambda^k.
\end{equation*}
Hereafter, $p$ is the adjoint variable associated with $u$ and it is obtained from the successive solution of the following two parabolic equations:
\begin{equation}\label{dis_state}
\frac{\partial y}{\partial t}-\nu \Delta y+a_0y=u\chi_{\mathcal{O}}~ \text{in}~ \Omega\times(0,T),  \quad
y=0 ~ \text{on}~ \Gamma\times(0,T), \quad
y(0)=\varphi,
\end{equation}
and
\begin{equation}\label{dis_adjoint}
-\frac{\partial p}{\partial t}-\nu \Delta p+a_0p=\gamma(y-y_d)~ \text{in}~ \Omega\times(0,T),\quad
p=0~ \text{on}~ \Gamma\times(0,T),\quad
p(T)=0.
\end{equation}
It is clear that the equation (\ref{dis_state}) is just the state equation (\ref{state_eqn}) and it can be characterized by the operator $S$ with $y=S(u)$. Furthermore, we denote by $S^*$ the adjoint operator of $S$. Then, it is easy to derive that $S^*: L^2(Q)\longrightarrow L^2(\mathcal{O})$ satisfies $p|_{\mathcal{O}}=S^*(\gamma(y-y_d))$, where $p$ is the solution of the adjoint equation (\ref{dis_adjoint}). Then, we obtain the following first-order optimality condition of the $u$-subproblem (\ref{ADMM_u}).
\begin{theorem}\label{oc_u}
	Let $u^{k+1}$ be the unique solution of the subproblem (\ref{ADMM_u}). Then, $u^{k+1}$ satisfies
	\begin{equation}\label{oc}
	Dj_k(u^{k+1})=u^{k+1}+ p^{k+1}|_{\mathcal{O}}+\beta(u^{k+1}-z^{k})-\lambda^k=0,
	\end{equation}
	{where $p^{k+1}$ is the adjoint variable associated with $u^{k+1}$.}
\end{theorem}

\subsubsection{Remark on $\beta$}

According to (\ref{oc}), $Dj_k(u^{k+1})$ consists of the minimization of $\tilde{J}(u)$ and the satisfaction of the constraint on the control variable. It is natural to consider choosing some value that is not different from $1$ for $\beta$ so that these two objectives can be well balanced.
Our numerical experiments show that, $\beta=2$ or $3$, is usually a good choice to generate robust and fast numerical performance. Also, because of this reason, we reformulate the original problem (\ref{Basic_Problem})--(\ref{state_eqn}) as (\ref{equi_model}) with a scaled objective functional $\tilde{J}(u)$. If no scaling is considered, it is easy to show that the optimality condition of the corresponding $u$-subproblem reads
\begin{equation}\label{oc_unscaled}
\alpha (u^{k+1}+ p^{k+1}|_{\mathcal{O}})+\beta(u^{k+1}-z^{k})-\lambda^k=0,
\end{equation}
and it implies that the penalty parameter $\beta$ should be close to $\alpha$ in order to balance the two objectives in (\ref{oc_unscaled}). Since $\alpha$ is generally very small (e.g., less than $10^{-3}$), $\beta$ is also forced to be small for this case. According to our numerical experiments, too small values of $\beta$ may easily cause some stability and round-off problems in numerical implementation, and they also easily result in unbalanced magnitudes for the primal variables $u$ and $z$, and the dual variable $\lambda$. All these issues are inclined to deteriorate convergence of the ADMM.

\subsubsection{Subproblem (\ref{ADMM_z})}

For the $z$-subproblem (\ref{ADMM_z}), notice that
\begin{equation*}
L_\beta(u^{k+1},z, \lambda^k)=\tilde{J}(u^{k+1})+I_{\mathcal{C}}(z)-(\lambda^k, u^{k+1}-z)
+\frac{\beta}{2}\|u^{k+1}-z\|^2,
\end{equation*}
which implies that
\begin{equation*}\label{z_sub}
z^{k+1}=\arg\min_{z\in U} I_{\mathcal{C}}(z)-(\lambda^k, u^{k+1}-z)
+\frac{\beta}{2}\|u^{k+1}-z\|^2.
\end{equation*}
Hence, $z^{k+1}$ is given by
\begin{equation}\label{soln_z}
z^{k+1}=P_{\mathcal{C}}(u^{k+1}-\frac{\lambda^k}{\beta}),
\end{equation}
where $P_{\mathcal{C}}(\cdot)$ denotes the projection onto the admissible set ${\mathcal{C}}$:
$$
P_{\mathcal{C}}(v):=\max\{a,\min\{v, b\}\},\forall v\in U.
$$

\subsection{Inexactness Criterion}\label{se:out_iter}

In this subsection, we propose an inexactness criterion that achieves the mentioned goals, and an inexact version of the ADMM (\ref{ADMM}) is obtained for the problem(\ref{Basic_Problem})--(\ref{state_eqn}). Various inexact versions of the ADMM in different settings can be found in the literature. For example, inexact versions of the ADMM for the generic case have been discussed in \cite{Eckstein1992douglas, Eckstein2017relative, Ng2011inexact, yuan2005improvement}. These works require summable conditions on the sequence of accuracy (represented in terms of either the absolute or relative errors). Such a condition forces the subproblems to be solved with increasing accuracy and requires specifying the accuracy (indeed an infinite series of constants) a prior; both are difficult to be realized practically. A particular inexact version is the so-called proximal ADMM in, e.g., \cite{bredies2015preconditioned, He2002new}, which adds appropriate quadratic terms to regularize the subproblems and may alleviate these subproblems for some cases by specifying the proximal terms appropriately. Because of the different and much more difficult setting in the problem (\ref{Basic_Problem})--(\ref{state_eqn}), however, a specific criterion tailored for the subproblem (\ref{ADMM_u}) should be found in order to solve it more efficiently.

Recall that the optimality condition of the $u$-subproblem (\ref{ADMM_u}) can be characterized by (\ref{oc}).
Since the $u$-subproblem (\ref{ADMM_u}) is strongly convex, the above necessary condition is also sufficient. Therefore, if $\tilde{u}\in U$ satisfies $D{j_k}(\tilde{u})=0$, then $\tilde{u}$ is the unique solution of the $u$-subproblem (\ref{ADMM_u}). To propose an inexactness criterion, we define $e_k(u)$ as
\begin{equation}\label{eu_oper}
e_{k}(u):= (1+\beta) u+S^*(\gamma(S(u)-y_d))-\beta z^{k}- \lambda^k.
\end{equation}
It follows from the definitions of the solution operator $S$ and its adjoint operator $S^*$ that $e_k(u)$ can be written as
\begin{equation}\label{def_eu}
e_{k}(u)= (1+\beta) u+p|_{\mathcal{O}}-\beta z^{k}- \lambda^k,
\end{equation}
{where $p$ is the adjoint variable associated with $u$.}

It is clear that $e_k(u)=Dj_k(u)$ and $u^{k+1}$ is the solution of the $u$-subproblem (\ref{ADMM_u}) at the $(k+1)$-th iteration if and only if $e_{k}(u^{k+1})=0$. Hence, we can use $e_{k}(u)$ as a residual for the $u$-subproblem (\ref{ADMM_u}). With the help of $e_{k}(u)$, we propose the following inexactness criterion. For a given constant $\sigma$ satisfying
\begin{equation}\label{sigma}
0<\sigma<\frac{\sqrt{2}}{\sqrt{2}+\sqrt{\beta}}\in(0,1),
\end{equation}
we compute $u^{k+1}$ such that
\begin{equation}\label{inexact_criterion}
\|e_{k}(u^{k+1})\|\leq\sigma\|e_{k}(u^{k})\|.
\end{equation}

The inexactness criterion (\ref{inexact_criterion}) is mainly inspired by our previous work \cite{yue2018implementing}, and it keeps all advantageous features of the criterion in \cite{yue2018implementing}. Meanwhile, the problem (\ref{Basic_Problem})--(\ref{state_eqn}) in an infinite-dimensional Hilbert space is much more complicated than the LASSO model considered in \cite{yue2018implementing}, and it is worthy to elaborate on the details of executing the inexactness criterion (\ref{inexact_criterion}). Indeed, the residual $e_k(u)$ in (\ref{def_eu}) is derived from the first-order derivative of $j_k(u)$. Conceptually, the computation of $e_k(u)$ requires the solutions of the state equation (\ref{state_eqn}) and the adjoint equation (\ref{dis_adjoint}). Practically, the residual $e_k(u)$ can be calculated easily by certain iterative scheme, see Algorithm \ref{AbstractCG} for the detail of implementing the CG method.

\begin{remark}\label{remark-criterion}
We reiterate that the inexactness criterion (\ref{inexact_criterion}) can be checked by current iterates and it can be executed automatically during iterations. There is no need to set any empirically perceived constant accuracy a prior, and it is independent of the mesh sizes for discretization. Also, the relative error $\|e_{k}(u^{k+1})\|/\|e_{k}(u^{k})\|$ is controlled by the constant $\sigma$ (instead of summable sequences as proposed in many ADMM literatures) and it does not need to tend to zero (hence, increasing accuracy can be avoided in iterations). All these features make the inexactness criterion (\ref{inexact_criterion}) easily implementable and more likely to save computation.
\end{remark}

\subsection{An Inexact Version of the ADMM (\ref{ADMM}) for (\ref{Basic_Problem})--(\ref{state_eqn})}

Based on the previous discussion, an inexact version of the ADMM (\ref{ADMM}) with the inexactness criterion (\ref{inexact_criterion}) can be proposed for the problem (\ref{Basic_Problem})--(\ref{state_eqn}).

\begin{algorithm}[h]\caption{{An Inexact Version of the ADMM (\ref{ADMM}) for (\ref{Basic_Problem})--(\ref{state_eqn})}}
	\begin{algorithmic}[0]
		\Require $\{u^0,z^0,\lambda^0\}^\top\in U\times U\times U$, $\beta>0$ and $0<\sigma<\frac{\sqrt{2}}{\sqrt{2}+\sqrt{\beta}}\in(0,1)$.
		\While{not converged}
		\State Compute
		$
		e_{k}(u^k)=(1+\beta)u^k+ p^k|_{\mathcal{O}}-\beta z^{k}-\lambda^k.
		$
		\State Find $u^{k+1}$ such that
		\begin{equation*}
		\|e_{k}(u^{k+1})\|\leq\sigma\|e_{k}(u^{k})\|,~\text{with}~e_{k}(u^{k+1})=(1+\beta)u^{k+1}+ p^{k+1}
		|_{\mathcal{O}}-\beta z^{k}-\lambda^k.
		\end{equation*}
		\State Update the variable $z^{k+1}$:
		$
		z^{k+1}=P_{\mathcal{C}}(u^{k+1}-\frac{\lambda^k}{\beta}).
		$
		\State Update the Lagrange multiplier $\lambda^{k+1}$:
		$
		\lambda^{k+1} = \lambda^k-\beta(u^{k+1}-z^{k+1}).
		$
		\EndWhile
	\end{algorithmic}\label{InADMM_algorithm}
\end{algorithm}

\section{Convergence Analysis}\label{sec:convergence}

In this section, we prove the strong global convergence for Algorithm \ref{InADMM_algorithm}. Though there are many works in the literature studying the convergence of the ADMM and its variants, the convergence of Algorithm \ref{InADMM_algorithm} should be proved from scratch because of the specific inexactness criterion (\ref{inexact_criterion}) and the setting of the problem (\ref{Basic_Problem})--(\ref{state_eqn}). In particular, the proof is essentially different from that in \cite{yue2018implementing}, despite of some common ideas in the respective stopping criteria. Note that the strong global convergence to be obtained is because of the strong convexity of the objective functional $\tilde{J}(u)$ in (\ref{equi_model}), which is usually absent for many other problems such as the LASSO model considered in \cite{yue2018implementing}.

\subsection{Preliminary}

To present our analysis in a compact form, we denote $w\in W:=U\times U\times U$, $v \in V:= U\times U$ and the function $F(w)$ as follows:
\begin{equation}\label{notation}
w = \begin{pmatrix} u\\ z \\ \lambda \end{pmatrix},
v=\begin{pmatrix} z\\ \lambda \end{pmatrix},\,
\hbox{and}~ F(w) = \begin{pmatrix}D\tilde{J}(u)- \lambda \\ \lambda \\ u-z \end{pmatrix},
\end{equation}
where $D\tilde{J}(u)$ is the first-order derivative of $\tilde{J}(u)$.
We also define the norm
\begin{equation}\label{Hnorm}
\|v\|_H=\sqrt{(v,Hv)}:=\sqrt{\beta\| z\|^2+\frac{1}{\beta}\|\lambda\|^2}, \quad\forall v\in V,
\end{equation}
which is induced by the matrix operator
\begin{equation*}
H=\begin{pmatrix}
\beta I & 0\\
0&\frac{1}{\beta}I
\end{pmatrix}.
\end{equation*}
With these notations, it is easy to see that the problem (\ref{equi_model}) can be characterized as the following variational inequality: find $w^*=\left( u^*,z^*,\lambda^*\right)^\top \in W$ such that
\begin{eqnarray}\label{Algo_General_VI}
{\hbox{VI}}(W,\mathcal{C}, F) \!\!\!\!&:&  I_{\mathcal{C}}(z) - I_{\mathcal{C}}(z^*) + ( w - w^* , F(w^*)) \ge 0,\  \forall w \in W.
\end{eqnarray}
We denote by $W^*$ the solution set of the variational inequality (\ref{Algo_General_VI}); and
it is easy to show that the solution set $W^*$ is a singleton.

From the definition of $\tilde{J}$ in (\ref{reduced_objective}), we know that it is strongly convex, i.e.
\begin{equation}\label{str_con_of_J}
\|u-v\|^2\le ( u - v , D\tilde{J}(u)- D\tilde{J}(v)),\; \forall u, v\in U.
\end{equation}
In addition, one can show that $D\tilde{J}$ is Lipschitz continuous. Indeed, one has
$$
D\tilde{J}(u)=u+p|_{\mathcal{O}},
$$
where $p$ is the adjoint variable associated with $u$. We introduce a linear operator $\bar{S}: U\longrightarrow Y$ such that
 \begin{equation}\label{def_barS}
S(v)=\bar{S}v+S(0),\quad\forall v\in U.
\end{equation}
Then, we can derive that
\begin{equation}\label{property_of_J}
( u - v , D\tilde{J}(u)- D\tilde{J}(v))\le \kappa \|u-v\|^2,\; \forall u, v\in U,
\end{equation}
where $\kappa=1+\gamma\|\bar{S}^*\bar{S}\|$.

\subsection{Optimality Conditions}
Recall that in Algorithm \ref{InADMM_algorithm}, the $u$-subproblem (\ref{ADMM_u}) is inexactly solved subject to the inexactness criterion (\ref{inexact_criterion}), and the $z$-subproblem (\ref{ADMM_z}) and $\lambda$-subproblem (\ref{ADMM_lambda}) can be solved exactly. Hence, for the sequence $w^{k+1}=(u^{k+1},z^{k+1},\lambda^{k+1})^\top$ generated by Algorithm \ref{InADMM_algorithm}, the first-order optimality conditions can be expressed as:
\begin{subequations}\label{optimality}
	\begin{numcases}
	~ D_{u} {L}_{\beta}( u^{k+1},z^k,\lambda^k )=e_k(u^{k+1}),\label{optimality_u}\\
	~I_{\mathcal{C}}(z) - I_{\mathcal{C}}(z^{k+1}) + ( z-z^{k+1},\lambda^k-\beta (u^{k+1}-z^{k+1}))\geq 0,\; \forall z\in U ,\label{optimality_z}\\
	~\lambda^{k+1}=\lambda^{k}-\beta(u^{k+1}-z^{k+1}),\label{optimality_lam}
	\end{numcases}
\end{subequations}
where $D_u{L}_{\beta}( u^{k+1},z^k,\lambda^k )$ is the first-order partial derivative of ${L}_{\beta}\left(u, z,\lambda \right)$ with respect to $u$ at $\left( u^{k+1},z^k,\lambda^k \right)^\top$.

To prove the convergence of Algorithm \ref{InADMM_algorithm}, it is crucial to analyze the residual $e_k (u^{k+1})$. It follows from (\ref{def_eu}) and (\ref{inexact_criterion}) that
\begin{eqnarray}\label{ek1}
\begin{aligned}
\| e_{k}(u^{k+1})\|&\le \sigma  \| e_{k}(u^{k})\|=\sigma\| e_{k-1}(u^{k})+\beta z^{k-1}+\lambda^{k-1}-\beta z^{k}-\lambda^k\|\\
&\le\sigma \| e_{k-1}(u^{k})\|+\sigma \|\beta z^{k-1}+\lambda^{k-1}-\beta z^{k}-\lambda^k\|.
\end{aligned}
\end{eqnarray}
In addition, it follows from (\ref{optimality_z}) that
\begin{equation}\label{b1}
I_{\mathcal{C}}(z^k) - I_{\mathcal{C}} (z^{k+1}) + ( z^k- z^{k+1},  \lambda^k - \beta ( u^{k+1}-z^{k+1} ) ) \ge 0,
\end{equation}	
and
\begin{equation}\label{b2}
I_{\mathcal{C}}(z^{k+1}) - I_{\mathcal{C}} (z^{k}) + ( z^{k+1}- z^{k},  \lambda^{k-1} - \beta ( u^{k}-z^{k} ) ) \ge 0.
\end{equation}
Adding (\ref{b1}) and (\ref{b2}) together, we have
\begin{equation}\label{convex_g}
( z^{k+1}-z^k,\lambda^{k+1}-\lambda^k) \le 0.
\end{equation}
Then, it follows from (\ref{ek1}) and (\ref{convex_g}) that
\begin{eqnarray}\label{Basic_equation_3}
\begin{aligned}
\| e_{k}(u^{k+1})\|
\le&\sigma \| e_{k-1}(u^{k})\|+\sigma\left(\|\beta z^{k-1}-\beta z^{k}\|^2+\|\lambda^{k-1}-\lambda^k\|^2\right)^{\frac{1}{2}}\\
=&\sigma \| e_{k-1}(u^{k})\|+\sigma \sqrt{\beta} \|v^{k}-v^{k-1}\|_{H}.
\end{aligned}
\end{eqnarray}
Moreover, we note that the condition (\ref{sigma})
implies that
$$
0<\frac{\beta}{2}\frac{\sigma^2}{(1-\sigma)^2}=\left(\frac{\sigma}{2(1-\sigma)}\right)\left(\frac{\beta\sigma}{1-\sigma}\right)<1,
$$
then there exits a constant $\mu>0$ such that
\begin{equation}\label{Theory_Contractive_6}
(1-\frac{\mu}{2}\frac{\sigma}{1-\sigma})>0 \quad\text{and}\quad (1-\frac{1}{\mu}\frac{\sigma}{1-\sigma}\beta)>0.
\end{equation}
These inequalities will be used later.

\subsection{Convergence}
With above preparations, we are now in a position to prove the convergence for Algorithm \ref{InADMM_algorithm}. To simplify the notation, let us introduce an auxiliary variable $\bar{w}^k$ as
\begin{equation}\label{wbar}
\bar{w}^k=\begin{pmatrix}\bar{u}^k\\ \bar{z}^{k}\\ \bar{\lambda}^k \end{pmatrix}=\begin{pmatrix}u^{k+1}\\ z^{k+1}\\ \lambda^k-\beta(u^{k+1}-z^k) \end{pmatrix}.
\end{equation}
The role of $\bar{w}^k$ is just for simplifying the notation in our analysis; it is not required to be computed for implementing Algorithm \ref{InADMM_algorithm}. Next, we prove some results which will be useful in the following discussion.

First of all, we analyze how different the point $\bar{w}^k$ defined in (\ref{wbar}) is from the solution $w^*$ of (\ref{Algo_General_VI}) and how to quantify this difference by iterates generated by Algorithm \ref{InADMM_algorithm}.
\begin{lemma}\label{P1}
	Let $\left\{ w^k \right\}=\{(u^k,z^k,\lambda^k)^\top\}$ be the sequence generated by Algorithm \ref{InADMM_algorithm} and $\{\bar{w}^k\}=\{(\bar{u}^k,\bar{z}^k,\bar{\lambda}^k)^\top\}$ be defined as in (\ref{wbar}). Then, for all $w\in W$, one has
\begin{eqnarray}
\begin{aligned}\label{Theory_Contractive_1}
&I_{\mathcal{C}}(\bar{z}^k)-I_{\mathcal{C}}(z)+ ( \bar{w}^k - w , F ( \bar{w}^k) )  \le~  \frac{1}{2} \left(\| v^k - v \|_{H}^2 - \| v^{k+1} - v \|_{H}^2 - \| v^{k} - v^{k+1} \|_{H}^2 \right)\\
&~~~~~~~~~~~~~~~~~~~~~~~~~~~~~~~~~~~~~~~~~~~~~  + \left( u^{k+1} - u , D_{u} {L}_{\beta}( u^{k+1},z^k,\lambda^k )\right).
\end{aligned}
\end{eqnarray}
\end{lemma}

\begin{proof}We first rewrite $D_{u} {L}_{\beta} \left( u^{k+1},u^k,\lambda^k \right)$ as
	\begin{eqnarray}
	D_{u} {L}_{\beta}( u^{k+1}, z^k, \lambda^k ) =D\tilde{J}(u^{k+1})-  ( \lambda^k - \beta ( u^{k+1} -z^k ) )
	= D\tilde{J}(u^{k+1})- \bar{\lambda}^k,\nonumber
	\end{eqnarray}
	with which we obtain, for all $w\in W$, that
	{\begin{eqnarray}
		\begin{aligned}\label{Basic_inequality1}
		&I_{\mathcal{C}}(z) -I_{\mathcal{C}}(\bar{z}^{k}) +( w - \bar{w}^k , F ( \bar{w}^k ) )\\
		=~& ( u - u^{k+1} , D\tilde{J}(u^{k+1}) -\bar{\lambda}^k )\\
		&+ I_{\mathcal{C}}(z) - I_{\mathcal{C}}(z^{k+1})+ ( z - z^{k+1} , \bar{\lambda}^k )+( \lambda - \bar{\lambda}^k , u^{k+1}-z^{k+1})\\
		=~& ( u - u^{k+1} , D_u {L}_{\beta} ( u^{k+1},z^k,\lambda^k ) )+ ( z - z^{k+1}, \lambda^k -\beta (u^{k+1} -z^{k+1}  ) ) \\
		&  + I_{\mathcal{C}}(z) - I_{\mathcal{C}}(z^{k+1}) + \beta ( z - z^{k+1} , z^{k} - z^{k+1} ) + \frac{1}{\beta}( \lambda - \bar{\lambda}^k , \lambda^k - \lambda^{k+1}) \\
		\overset{(\ref{optimality_z})}{\ge}& ( u - u^{k+1} , D_u {L}_{\beta} ( u^{k+1},z^k,\lambda^k) )+ \beta ( z - z^{k+1} , z^{k} - z^{k+1} )\\
		&+ \frac{1}{\beta}(  \lambda - \lambda^{k+1}, \lambda^k - \lambda^{k+1}) + \frac{1}{\beta} (  \lambda^{k+1} - \bar{\lambda}^k , \lambda^k - \lambda^{k+1} ).
		\end{aligned}
		\end{eqnarray}}
Applying the identity
	\begin{equation}\label{Basic_equation_{L^2(O_r)}}
	( a - c , b - c) = \frac{1}{2} \left( \| a - c \|^2 - \| a - b \|^2 + \| b - c \|^2 \right)
	\end{equation}
	to (\ref{Basic_inequality1}), we have
	{\small
		{\begin{eqnarray}
			&& I_{\mathcal{C}}(z) - I_{\mathcal{C}}(\bar{z}^{k}) +( w - \bar{w}^k , F ( \bar{w}^k ) )\nonumber \\
			&\overset{(\ref{Basic_equation_{L^2(O_r)}})}{\ge}& ( u - u^{k+1} , D_u {L}_{\beta} ( u^{k+1},z^k,\lambda^k ) )
			+ \frac{\beta}{2} \left(\| z - z^{k+1} \|^2 - \| z - z^{k}  \|^2 + \| z^k - z^{k+1}  \|^2 \right) \nonumber \\
			& & \quad + \frac{1}{2\beta} \left( \| \lambda - \lambda^{k+1} \|^2 - \|  \lambda - \lambda^{k} \|^2 + \| \lambda^k - \lambda^{k+1} \|^2 \right) - ( z^k - z^{k+1}, \lambda^k - \lambda^{k+1} ) \nonumber \\
			&\overset{(\ref{convex_g})}{\ge}&( u - u^{k+1} , D_u {L}_{\beta} ( u^{k+1},z^k,\lambda^k ) )
			+ \frac{\beta}{2} \left(\| z - z^{k+1} \|^2 - \| z - z^{k}  \|^2 + \| z^k - z^{k+1}  \|^2 \right) \nonumber \\
			& & \quad + \frac{1}{2\beta} \left( \| \lambda - \lambda^{k+1} \|^2 - \|  \lambda - \lambda^{k} \|^2 + \| \lambda^k - \lambda^{k+1} \|^2 \right), \quad \forall w \in W. \label{Theory_Contractive_3}
			\end{eqnarray}}
	}
Using the definition of $H$-norm in (\ref{Hnorm}), the result (\ref{Theory_Contractive_3}) can be rewritten as (\ref{Theory_Contractive_1}) and the proof is complete.\end{proof}
The difference between the inequality (\ref{Theory_Contractive_1}) and the variational inequality reformulation (\ref{Algo_General_VI}) reflects the difference of the point $\bar{w}^k$
from the solution point $w^*$. For the right-hand side of (\ref{Theory_Contractive_1}), the first three terms are quadratic and they are easy to manipulate over different indicators by algebraic operations, but it is not that explicit how the last crossing term can be controlled towards the eventual goal of proving the convergence of the sequence $\{w^k\}$. We thus look into this term particularly and show that the sum of these crossing terms over $K$ iterations can be bounded by some quadratic terms as well. This result is summarized in the following lemma.
\begin{lemma}\label{P2}
	Let $\left\{ w^k \right\}=\{(u^k,z^k,\lambda^k)^\top\}$ be the sequence generated by Algorithm \ref{InADMM_algorithm}. For any integer $K>0$ and $\mu$ satisfying (\ref{Theory_Contractive_6}), one has
	{ \begin{equation}\label{Key_3}
	\begin{aligned}
	\sum_{k=1}^K( u^{k+1} - u , D_{u} {L}_{\beta}( u^{k+1},z^k,\lambda^k ))
	\le  \frac{\mu}{2}\sum_{k=1}^K\frac{\sigma}{1-\sigma}\| u^{k+1}-u\|^2+ \frac{1}{2\mu}\sum_{i=1}^{K-1}\frac{\sigma}{1-\sigma} \beta \|v^{i}-v^{i+1}\|_{H}^2\\
	+\frac{1}{2\mu} \frac{\sigma}{1-\sigma} \left[\| e_{0}(u^{1})\|+\sqrt{\beta} \|v^0-v^1\|_{H} \right]^2, \forall u\in U.
	\end{aligned}
	\end{equation}}
\end{lemma}
\begin{proof}
	First, it follows from (\ref{Basic_equation_3}) that
	\begin{equation}\label{sum_eu}
	\| e_{k}(u^{k+1})\|\le \sum_{i=0}^{k-1} \sigma^{k-i} \sqrt{\beta}\|v^{i}-v^{i+1}\|_{H}+\sigma^k\| e_{0}(u^{1})\|.
	\end{equation}
From (\ref{optimality_u}) and (\ref{sum_eu}), for any $\mu>0$ satisfying (\ref{Theory_Contractive_6}) and $u \in U$, we have
\begin{eqnarray*}
	&& \quad\sum_{k=1}^K ( u^{k+1} - u , D_u {L}_{\beta}( u^{k+1},z^k,\lambda^k )) \le \sum_{k=1}^K \| u^{k+1}-u\| \|  e_{k}(u^{k+1}) \| \\
	&&\le  \sum_{k=1}^K \sum_{i=0}^{k-1}  \sigma^{k-i} \sqrt{\beta} \| u^{k+1}-u\| \|v^{i}-v^{i+1}\|_{H}+\sum_{k=1}^K   \sigma^k\| u^{k+1}-u\| \| e_{0}(u^{1})\|\\
	&&\le  \sum_{k=1}^K \sum_{i=1}^{k-1}  \sigma^{k-i} \sqrt{\beta} \| u^{k+1}-u\| \|v^{i}-v^{i+1}\|_{H}+\sum_{k=1}^K   \sigma^k \| u^{k+1}-u\| \left[\| e_{0}(u^{1})\|+\sqrt{\beta}\|v^0-v^1\|_{H} \right]\\
	&&\le  \frac{\mu}{2}\sum_{k=1}^K \sum_{i=1}^{k-1}  \sigma^{k-i}  \| u^{k+1}-u\|^2+  \frac{1}{2\mu} \sum_{k=1}^K \sum_{i=1}^{k-1}  \sigma^{k-i} \beta \|v^{i}-v^{i+1}\|_{H}^2\\
	&&\quad+ \frac{\mu}{2}\sum_{k=1}^K   \sigma^k \| u^{k+1}-u\|^2+ \frac{1}{2\mu} \sum_{k=1}^K   \sigma^k  \left[\| e_{0}(u^{1})\|+\sqrt{\beta} \|v^0-v^1\|_{H} \right]^2\\
	&&=  \frac{\mu}{2}\sum_{k=1}^K \sum_{i=0}^{k-1}  \sigma^{k-i}  \| u^{k+1}-u\|^2+  \frac{1}{2\mu} \sum_{k=1}^K \sum_{i=1}^{k-1}  \sigma^{k-i} \beta \|v^{i}-v^{i+1}\|_{H}^2\\
	&&\quad+  \frac{1}{2\mu} \sum_{k=1}^K   \sigma^k  \left[\| e_{0}(u^{1})\|+\sqrt{\beta} \|v^0-v^1\|_{H} \right]^2.
\end{eqnarray*}
	Then, we have
	{
		\begin{equation*}
		\begin{aligned}
		&\sum_{k=1}^K \left( u^{k+1} - u , D_{u} {L}_{\beta}( u^{k+1},z^k,\lambda^k )\right) \\
		\le~& \frac{\mu}{2}\sum_{k=1}^K\frac{\sigma-\sigma^{k+1}}{1-\sigma}\| u^{k+1}-u\|^2+ \frac{1}{2\mu}\sum_{i=1}^{K-1}\frac{\sigma-\sigma^{K-i+1}}{1-\sigma} \beta \|v^{i}-v^{i+1}\|_{H}^2\\
		&+\frac{1}{2\mu} \frac{\sigma-\sigma^{K+1}}{1-\sigma} \left[\| e_{0}(u^{1})\|+\sqrt{\beta} \|v^0-v^1\|_{H} \right]^2\\
		\le~&  \frac{\mu}{2}\sum_{k=1}^K\frac{\sigma}{1-\sigma}\| u^{k+1}-u\|^2+ \frac{1}{2\mu}\sum_{i=1}^{K-1}\frac{\sigma}{1-\sigma} \beta \|v^{i}-v^{i+1}\|_{H}^2\\
		&+\frac{1}{2\mu} \frac{\sigma}{1-\sigma} \left[\| e_{0}(u^{1})\|+\sqrt{\beta} \|v^0-v^1\|_{H} \right]^2, \quad \forall u \in U.
		\end{aligned}
		\end{equation*}
	}
	We thus complete the proof.
\end{proof}
Now we can  establish the strong global convergence of Algorithm \ref{InADMM_algorithm}.
\begin{theorem}\label{thm-convergence}
	Let $w^*=(u^*,z^*,\lambda^*)^\top$ be the solution point of the variational inequality (\ref{Algo_General_VI}) and $\left\{ w^k \right\}=\{(u^k, z^k,\lambda^k)^\top\}$ be the sequence generated by Algorithm \ref{InADMM_algorithm}. Then, we have the following assertions:
	\begin{enumerate}
		\item[(1)] $\|e_{k} (u^{k+1}) \|\overset{k\rightarrow \infty}{\longrightarrow} 0,\quad \| z^{k} - z^{k+1}
		\| \overset{k\rightarrow \infty}{\longrightarrow} 0$,\quad $\|u^{k+1}-z^{k+1}\| \overset{k\rightarrow \infty}{\longrightarrow} 0$;
		\item[(2)] $u^{k}\overset{k\rightarrow \infty}{\longrightarrow}u^*$, $z^{k}\overset{k\rightarrow \infty}{\longrightarrow}z^*$ and $\lambda^{k}\overset{k\rightarrow \infty}{\longrightarrow}\lambda^*$ strongly in $U$.
	\end{enumerate}
\end{theorem}
\begin{proof}
	(1). First, it follows from (\ref{notation}), (\ref{str_con_of_J}) and (\ref{wbar}) that
	\begin{equation}\label{Basic_equation_8}
	( w - \bar{w}^k, F\left(w\right) - F (\bar{w}^k)) =( u - \bar{u}^k,D\tilde{J}(u)-D\tilde{J}(\bar{u}^k))\ge \|  u - u^{k+1} \|^2.
	\end{equation}
Then, using the results (\ref{Theory_Contractive_1}) and (\ref{Key_3}) established in Lemma \ref{P1} and Lemma \ref{P2}, respectively, we obtain
	{\begin{eqnarray}
		&&\sum^{K}_{k=1} \left\{ I_{\mathcal{C}} ( \bar{z}^k) -I_{\mathcal{C}}(z) +( \bar{w}^k - w,  F(w)) \right\}\nonumber\\
		&=& \sum^{K}_{k=1} \left\{ I_{\mathcal{C}} ( \bar{z}^k) - I_{\mathcal{C}}\left( z \right) +( \bar{w}^k - w,  F(\bar{w}^k))+ ( \bar{w}^k-w, F\left(w\right) - F (\bar{w}^k)) \right\}\nonumber\\
		&\overset{(\ref{Theory_Contractive_1})}{\le}& \frac{1}{2} \left( \| v^{1} - v \|_{H}^2 - \| v^{K + 1} - v \|_{H}^2 \right) + \sum^{K}_{k=1} \left\{( u^{k+1} - u, D_u {L}_{\beta}( u^{k+1},z^k,\lambda^k)) \right. \nonumber\\
		&&\left. - ( w - \bar{w}^k, F\left(w\right) - F (\bar{w}^k))\right\} - \sum_{k=1}^{K} \frac{1}{2}\| v^k - v^{k+1} \|_{H}^2\nonumber\\
		&\overset{(\ref{Key_3})(\ref{Basic_equation_8})}{\le}& \frac{1}{2} \left( \| v^{1} - v \|_{H}^2 - \| v^{K+1} - v \|_{H}^2 \right) + \sum^{K}_{k=1} \left( \frac{\mu}{2} \frac{\sigma}{1 - \sigma} - 1 \right) \| u^{k+1}-u \|^2  \nonumber\\
		&& + \sum^{K-1}_{k=1} \frac{1}{2} \left( \frac{\sigma}{1 - \sigma} \frac{\beta}{\mu} - 1 \right) \| v^{k} - v^{k+1} \|_{H}^2 - \frac{1}{2} \| v^{K} - v^{K+1} \|_{H}^2 \nonumber\\
		&& + \frac{1}{2\mu} \frac{\sigma}{1 - \sigma} \left( \| e_{0} (u^{1}) \| + \sqrt{\beta} \| v^{0} - v^{1} \|_{H} \right)^2,\; \forall w\in W.\label{Theory_Contractive_5}
		\end{eqnarray}}

For the solution point $w^*$, we have
	$$
	I_{\mathcal{C}}( \bar{z}^k ) - I_{\mathcal{C}}\left(z^*\right) + ( \bar{w}^k - w^*, F \left(w^*\right)) \ge 0,\;\forall k\ge 1.
	$$
	Setting $w = w^*$ in (\ref{Theory_Contractive_5}), together with the above property, for any integer $K > 1$, we have
	{{\begin{equation}\label{Theory_Contractive_7}
			\begin{aligned}
			&\sum^{K}_{k=1} \left( 1 - \frac{\mu}{2} \frac{\sigma}{1 - \sigma} \right) \| u^{k+1} - u^* \|^2 + \sum^{K-1}_{k=1} \left( \frac{1}{2} - \frac{\beta}{2\mu} \frac{\sigma}{1 - \sigma} \right) \| v^{k} - v^{k+1} \|_{H}^2\\
			\le~& \frac{1}{2} \| v^{1} - v^* \|_{H}^2 + \frac{1}{2\mu} \frac{\sigma}{1 - \sigma} \left( \| e_{0} (u^{1}) \| + \sqrt{\beta} \| v^{0} - v^{1} \|_{H} \right)^2 \\
			&- \frac{1}{2} \| v^{K+1} -v^* \|_{H}^2 - \frac{1}{2} \| v^{K} -v^{K+1} \|_{H}^2.
			\end{aligned}
			\end{equation}}}
	It follows from (\ref{Theory_Contractive_6}) that
	$$
	(1-\frac{\mu}{2}\frac{\sigma}{1-\sigma})>0 \quad\text{and}\quad (1-\frac{1}{\mu}\frac{\sigma}{1-\sigma}\beta)>0.
	$$
	Then, the inequality (\ref{Theory_Contractive_7}) implies
	\begin{equation}\label{converge1}
	\|u^{k+1}-u^* \|\overset{k\rightarrow\infty}{\longrightarrow}0
	~ \text{and}~\| v^{k+1}-v^k \|_{H}\overset{k\rightarrow\infty}{\longrightarrow}0,
	\end{equation}
	For any $\varepsilon>0$, there exists $k_0$, such that for all $k\geq k_0$, we have
	$\|v^{k+1}-v^k\|_{H}<\varepsilon$ and $\sigma^k<\varepsilon.$
	Then, for all $k\geq k_0$, it follows from (\ref{sum_eu}) that
	\begin{equation*}
	\begin{aligned}
	\quad\| e_{k}(u^{k+1})\|
	&\le \sum_{i=0}^{k-1} \sigma^{k-i} \sqrt{\beta} \|v^{i}-v^{i+1}\|_{H}+\sigma^k\| e_{0}(u^{1})\|\\
	&= \sum_{i=0}^{k_0-1} \sigma^{k-i} \sqrt{\beta} \|v^{i}-v^{i+1}\|_{H}+\sum_{i=k_0}^{k-1} \sigma^{k-i} \sqrt{\beta} \|v^{i}-v^{i+1}\|_{H}+\sigma^k\| e_{0}(u^{1})\|\\
	&\le \big(\sqrt{\beta}\max_{0\le i\le k_0-1}\|v^{i}-v^{i+1}\|_{H}\sum_{i=0}^{k_0-1} \sigma^{k-k_0-i} \big)\sigma^{k_0}+\sigma^k\| e_{0}(u^{1})\|\\
	&\quad+\big(\sqrt{\beta}\max_{k_0\le i\le k-1}\|v^{i}-v^{i+1}\|_{H}\sum_{i=k_0}^{k-1} \sigma^{k-i} \big)\\
	&\le \varepsilon\big[\sqrt{\beta}\max_{0\le i\le k_0-1}\|v^{i}-v^{i+1}\|_{H}\sum_{i=0}^{k_0-1} \sigma^{k-k_0-i} +\sqrt{\beta}\sum_{i=k_0}^{k-1} \sigma^{k-i}+\| e_{0}(u^{1})\|\big],
	\end{aligned}
	\end{equation*}
	which implies that
	\begin{equation*}
	\| e_{k}(u^{k+1})\|\overset{k\rightarrow +\infty}{\longrightarrow}0.
	\end{equation*}
	In addition, since $\| v^{k+1}-v^k \|_{H}\overset{k\rightarrow\infty}{\longrightarrow}0$, we conclude that
	\begin{equation*}
	\| z^{k+1}-z^k \|\overset{k\rightarrow\infty}{\longrightarrow}0\quad\hbox{and}\quad\| \lambda^{k+1}-\lambda^k \|\overset{k\rightarrow\infty}{\longrightarrow}0.
	\end{equation*}
	Then, from $\| u^{k+1}-z^{k+1} \|=\frac{1}{\beta}\| \lambda^{k+1}-\lambda^k \|$, we have
	$\| u^{k+1}-z^{k+1} \|\overset{k\rightarrow\infty}{\longrightarrow}0.$

	\noindent(2). From (\ref{converge1}), we know that $u^{k}\overset{k\rightarrow \infty}{\longrightarrow}u^*$ strongly in $U$. Combining with $\| u^{k+1}-z^{k+1} \|\overset{k\rightarrow\infty}{\longrightarrow}0$, one has $z^{k}\overset{k\rightarrow \infty}{\longrightarrow}z^*$ strongly in $U$.  From (\ref{Algo_General_VI}), it is easy to verify that $\lambda^*=D\tilde{J}(u^*)$.
	On the other hand, one has:
	$$
	\lambda^{k}= D\tilde{J}(u^{k+1})+\beta(u^{k+1}-z^k)-e_k(u^{k+1}).
	$$
	We thus have
	$$
	\lambda^{k}-\lambda^*=D\tilde{J}(u^{k+1})-D\tilde{J}(u^*)+\beta(u^{k+1}-u^k)+\beta(u^{k}-z^k)-e_k(u^{k+1}).
	$$
	Noting that $u^{k}\overset{k\rightarrow \infty}{\longrightarrow}u^*$, $u^{k}-z^k\overset{k\rightarrow \infty}{\longrightarrow}0$, $e_k(u^{k+1})\overset{k\rightarrow \infty}{\longrightarrow}0$ and $D\tilde{J}$ is Lipschitz continuous (see (\ref{property_of_J})), we have
	$$
	\lambda^{k}\overset{k\rightarrow \infty}{\longrightarrow}\lambda^*~\text{strongly in } U.
	$$		
	We thus complete the proof.
\end{proof}
\begin{remark}
	Clearly, it follows from Theorem \ref{thm-convergence} that the state variable $y^k=S(u^k)$ also converges strongly in $Y$ to $y^*=S(u^*)$ since $S$ is continuous.
\end{remark}

\begin{remark}
Note that the convergence analysis for Algorithm \ref{InADMM_algorithm} does not depend on how the inexactness criterion (\ref{inexact_criterion}) is satisfied and what the specific form of the solution operator $S$ is.
\end{remark}

\section{Convergence Rate}\label{sec:convergence_rate}
In \cite{he20121, he2015non}, the ADMM's $O(1/K)$ worst-case convergence rate in both the ergodic and non-ergodic senses have been initiated in the context of convex optimization with consideration of the Euclidean space, where $K$ denotes the iteration counter. Recall that an $O(1/K)$ worst-case convergence rate means that an iterate, whose accuracy to the solution under certain criterion is of the order $O(1/K)$, can be found after $K$ iterations of an iterative scheme. It can be alternatively explained as that it requires at most ${O}({1}/{\varepsilon})$ iterations to find an approximate
solution with an accuracy of $\varepsilon$. This type of convergence rate is in the worst-case nature, and it provides a worst-case but universal estimate on the speed of convergence. Hence, it does not contradict with some much faster speeds which might be witnessed empirically for a specific application (as to be shown in Section \ref{sec:numerical}).
In this section, we extend these results to Algorithm \ref{InADMM_algorithm} in an infinite-dimensional Hilbert space. Despite the more complicated settings, their proofs are similar to those in \cite{he20121, he2015non} and hence omitted.

\subsection{Ergodic Convergence Rate}
In this subsection, we follow \cite{he20121} to establish an $O(1/K)$ worst-case convergence rate in the ergodic sense for Algorithm \ref{InADMM_algorithm}.
We first introduce a criterion to measure the accuracy of an approximation of the variational inequality (\ref{Algo_General_VI}).
As analyzed in \cite{facchinei2007finite, he20121}, the solution set $W^*$ of the variational inequality (\ref{Algo_General_VI}) has the following characterization.
\begin{theorem}[cf. \cite{facchinei2007finite}]
	Let $W^*$ be the solution set of the variational inequality (\ref{Algo_General_VI}). Then, we have
	\begin{equation*}
	W^* = \bigcap_{w\in W}\ \left\{\hat{w}\in W:  I_{\mathcal{C}}(z) - I_{\mathcal{C}}( \hat{z})+ ( w - \hat{w}, F({w})) \ge 0 \right\}.
	\end{equation*}
\end{theorem}
The above result indicates that $\hat{w}\in W$ is an approximate solution of the variational inequality (\ref{Algo_General_VI}) with an accuracy of $\varepsilon>0$ if
\begin{equation}\label{ap_soln}
I_{\mathcal{C}}\left(\hat{z}\right) - I_{\mathcal{C}}\left(z\right) +( \hat{w} - w , F \left(w\right) )\le \varepsilon.
\end{equation}
Next, we show an $O(1/K)$ worst-case convergence rate for Algorithm \ref{InADMM_algorithm}.

\begin{theorem}\label{thm-ergodic}
	Let $\left\{ w^k \right\}=\{(u^k,z^k,\lambda^k)^\top\}$ be the sequence generated by Algorithm \ref{InADMM_algorithm}; and $\{\bar{w}^k\}=\{(\bar{u}^k,\bar{z}^k,\bar{\lambda}^k)^\top\}$ be defined as in (\ref{wbar}). For any integer $K \geq 1$, we further define
	\begin{equation}\label{what}
	\hat{w}_K = \frac{1}{K}\sum_{k=1}^{K}\bar{w}^k.
	\end{equation}
	Then, for all $w\in W$, one has
	\begin{eqnarray*}
		I_{\mathcal{C}}\left(\hat{z}_K\right) - I_{\mathcal{C}}\left(z\right) +( \hat{w}_K - w , F \left(w\right) )
		\le  \frac{1}{K} \left[ \frac{1}{2\mu} \frac{\sigma}{1 - \sigma} \left( \| e_{0} (u^{1}) \| + \sqrt{\beta} \| v^{0} - v^{1} \|_{H} \right)^2 + \frac{1}{2} \|v^0 - v \|_{H}^2\right].
	\end{eqnarray*}
\end{theorem}

This theorem shows that after $K$ iterations, we can find an approximate solution of the variational inequality (\ref{Algo_General_VI}) with an accuracy of $O(1/K)$.
This approximate solution is given in (\ref{what}), and it is the average of all the points $w^k$ which can be computed by all the known iterates generated by Algorithm \ref{InADMM_algorithm}. Hence, this is an $O(1/K)$ worst-case convergence rate in the ergodic sense for Algorithm \ref{InADMM_algorithm}.

\subsection{Non-ergodic Convergence Rate}

In this subsection, we extend the result in \cite{he2015non} to show an $O(1/K)$ worst-case convergence rate in the non-ergodic sense for Algorithm \ref{InADMM_algorithm}.

We first need to clarify a criterion to precisely measure the accuracy of an iterate to a solution point. It follows from (\ref{notation}) and (\ref{optimality}) that for the iterate $(u^{k+1},z^{k+1},\lambda^{k+1})^\top$ generated by Algorithm \ref{InADMM_algorithm}, for all $w \in W$, one has
{$$
	I_{\mathcal{C}}(z) - I_{\mathcal{C}}(z^{k+1})  + \left( w - w^{k+1}, F(w^{k+1})) +  (  w - w^{k+1},  \left( \begin{array}{c}
	-\beta  \left( z^k - z^{k+1}\right) - e_{k} (u^{k+1})  \\
	0\\
	\frac{1}{\beta}\left( \lambda^{k+1} - \lambda^k \right)
	\end{array} \right)\right) \ge 0.
	$$}
\hspace{-0.15cm}Taking (\ref{Algo_General_VI}) into account, we can show that $(u^{k+1},z^{k+1},\lambda^{k+1})^\top$ is a solution point of (\ref{Algo_General_VI}) if and only if $ \| v^{k} - v^{k+1} \|_{H}^2 = 0 $ and $\| e_{k} (u^{k+1})  \|^2 = 0$. Hence, it is reasonable to measure the accuracy of the iterate $(u^{k+1},z^{k+1},\lambda^{k+1})^\top$ by $ \| v^{k} - v^{k+1} \|_{H}^2 $ and $\| e_{k} (u^{k+1})  \|^2 $. Our purpose is thus to show that after $K$ iterations of Algorithm \ref{InADMM_algorithm}, both $ \| v^{k} - v^{k+1} \|_{H}^2 $ and $\| e_{k} (u^{k+1})  \|^2 $ can be bounded by upper bounds in order of $O(1/K)$.

\begin{theorem}\label{thm-nonergodic}
	Let $\left\{ w^k \right\}=\{(u^k,z^k,\lambda^k)^\top\}$ be the sequence generated by Algorithm \ref{InADMM_algorithm}. Then, for any integer $K \geq 1$, we have
	{\small\begin{align}\label{itercomplexity-nonergodic_1}
		\setlength{\abovedisplayskip}{0pt}%
		\setlength{\belowdisplayskip}{0pt}%
		\min_{1\le k\le K} \left\{ \| v^{k} - v^{k+1} \|_{H}^2 \right\} \le \frac{1}{K} \left[ \frac{1}{\mu_0} \| v^{1} - v^* \|_{H}^2 + \frac{1}{\mu_0\mu} \frac{\sigma}{1 - \sigma} \left( \| e_{0} (u^{1}) \| +\sqrt{\beta} \| v^{0} - v^{1} \|_{H} \right)^2 \right],
		\end{align}}
	and
	{\small{\begin{eqnarray}
			\min_{1\le k \le K} \left\{ \| e_{k} (u^{k+1}) \|^2 \right\}\nonumber
			&\le& \frac{2}{K} \left\{ \left( \frac{\sigma}{1 - \sigma} \sqrt{\beta}\right)^2  \left[ \frac{1}{\mu_0} \| v^{1} - v^* \|_{H}^2 + \frac{1}{\mu_0\mu} \frac{\sigma}{1 - \sigma} \left( \| e_{0} (u^{1}) \| + \sqrt{\beta} \| v^{0} - v^{1} \|_{H} \right)^2 \right]  \right\} \nonumber\\
			& &\quad + \frac{2}{K^2}\left[ \left( \frac{\sigma}{1 - \sigma} \right)^2\cdot\left(\| e_{0} (u^{1}) \| +\sqrt{\beta}\| v^{0} - v^{1} \|_{H} \right)^2\right],\label{itercomplexity-nonergodic_2}
			\end{eqnarray}}}
	where $w^*$ is the solution point, $\mu$ satisfies (\ref{Theory_Contractive_6}) and $\mu_0 = 1 - \frac{\beta}{\mu} \frac{\sigma}{1 - \sigma} > 0$.
\end{theorem}

We note that both values in the right-hand sides of (\ref{itercomplexity-nonergodic_1}) and (\ref{itercomplexity-nonergodic_2}) are order of ${O}({1}/{K})$. Therefore, this theorem provides an ${O}({1}/{K})$ worst-case convergence rate in the non-ergodic sense for Algorithm \ref{InADMM_algorithm}.



\section{Implementation of Algorithm \ref{InADMM_algorithm}} \label{sec:implementation}

In this section, we discuss how to execute the inexactness criterion (\ref{inexact_criterion}) so as to specify Algorithm \ref{InADMM_algorithm} as a concrete algorithm for the problem (\ref{Basic_Problem})--(\ref{state_eqn}), and delineate the implementation details.

Indeed, the $u$-subproblem (\ref{ADMM_u}) is a typical unconstrained parabolic optimal control problem and various numerical methods in the literature can be applied. Whichever such method is applied, we should and only need to ensure that the inexactness criterion (\ref{inexact_criterion}) is satisfied in order to guarantee the overall convergence of Algorithm \ref{InADMM_algorithm}. Below we illustrate by the CG method how to execute the inexactness criterion (\ref{inexact_criterion}) in the inner-layer iterations. Recall that the $u$-subproblem (\ref{ADMM_u}) is
\begin{equation}\label{u_CG}
u^{k+1}=\arg\min_{u\in U}\;j_k(u)=\frac{\gamma}{2} \|S(u)-y_d\|^2+\frac{1}{2}\|u\|^2-(\lambda^k,u-z^{k})+\frac{\beta}{2}\|u-z^{k}\|^2,
\end{equation}
and the associated optimality condition is given in Theorem \ref{oc_u}. Next, we show that the optimality condition of the problem (\ref{u_CG}) can be characterized by a symmetric and positive definite linear system, hence the CG method can be applied. To this end, we first recall that the linear operator $\bar{S}$ defined in (\ref{def_barS}) satisfies
\begin{equation*}
S(v)=\bar{S}v+S(0),\quad\forall v\in U.
\end{equation*}
Then, $y=\bar{S}u$ is equivalent to the following equation:
$$
\frac{\partial y}{\partial t}-\nu \Delta y+a_0y=u\chi_{\mathcal{O}}~ \text{in}~ \Omega\times(0,T), \quad
y=0~\text{on}~ \Gamma\times(0,T),\quad
y(0)=0.
$$
In addition, it is easy to show that the adjoint operator $\bar{S}^*: Y\longrightarrow U$ satisfies $\bar{S}^*y=p|_{\mathcal{O}}$, where $p$ solves
$$
-\frac{\partial p}{\partial t}
-\nu\Delta p +a_0p=y~ \text{in}~ \Omega\times(0,T), \quad
p=0~
\text{on}~ \Gamma\times(0,T),\quad
p(T)=0.
$$
Hence, the $u$-subproblem (\ref{u_CG}) can be reformulated as
\begin{equation*}
u^{k+1}=\arg\min_{u\in U}\;j_k(u)=\frac{\gamma}{2} \|\bar{S}u+S(0)-y_d\|^2+\frac{1}{2}\|u\|^2-(\lambda^k,u-z^{k})+\frac{\beta}{2}\|u-z^{k}\|^2,
\end{equation*}
and the corresponding optimality condition is
\begin{equation}\label{spd_u}
(1+\beta+\gamma\bar{S}^*\bar{S})u^{k+1}+\gamma\bar{S}^*(S(0)-y_d)-\lambda^k-\beta z^k=0.
\end{equation}
Note that (\ref{spd_u}) is a symmetric and positive definite linear system of $u^{k+1}$ and the CG method can be applied. Obviously, at each iteration of Algorithm 1, we need to solve a linear system discretized from (\ref{spd_u}), with the same coefficient matrix, but different right-hand sides. Hence, a uniform preconditioner can be applied when certain iterative method (e.g., CG method) is employed to solve these linear systems. Recall that if SSN methods are applied, the coefficient matrices of the resulting Newton systems vary iteratively and preconditioners should also be chosen iteratively. This is a major difference of the ADMM from SSN methods for the problem (\ref{Basic_Problem})--(\ref{state_eqn}) .

\begin{algorithm}[h]
	\caption{CG for the $u$-subproblem (\ref{ADMM_u})}
	\begin{algorithmic}[0]
		\State \textbf{Input} $u_0^k=u^k, p_0^k=p^k$. Compute $g_0^k=u_0^k+ p_0^k|_{\mathcal{O}}+\beta(u_0^k-z^{k})-\lambda^k$, set $w_0^k=g_0^k$ and $e_k(u^k)=g_0^k$.
		\While{$\|e_{k}(u_{m}^{k})\|>\sigma\|e_{k}(u^{k})\|$}
		\State  Solving
		$\bar{y}_m^k=\bar{S}w_m^k
		$ and $ \bar{p}_m^k|_{\mathcal{O}}=\bar{S}^*(\gamma\bar{y}_m^k)$. Then compute the step size:
		$$
		\rho_m^k=\frac{(g_m^k, w_m^k)}{(\bar{g}_m^k, w_m^k)},\quad\text{with}\quad \bar{g}_m^k=(1+\beta)w_m^k+{\bar{p}_m^k}|_{\mathcal{O}}.
		$$
		\State Update $u, p$, the gradient $g$ and the residual $e_k(u_{m+1}^k)$ via:
		\begin{eqnarray*}
			&u_{m+1}^k=u_m^k-\rho_m^k w_m^k,
			&p_{m+1}^k=p_m^k-\rho_m^k \bar{p}_m^k,\\
			&g_{m+1}^k=g_m^k-\rho_m^k\bar{g}_m^k,
			&e_{k}(u_{m+1}^{k})= g_{m+1}^k.
		\end{eqnarray*}
		\State Compute $r_m^k=\|g_{m+1}^k\|^2/\|g_{m}^k\|^2,$
		and then update
		$w_{m+1}^k=g_{m+1}^k+r_m^k w_m^k.$
		\EndWhile
		\State \textbf{Output} $u^{k+1}=u^k_{m+1}$ and $p^{k+1}=p^k_{m+1}$.
	\end{algorithmic}\label{AbstractCG}
\end{algorithm}
With the inexactness criterion (\ref{inexact_criterion}), the CG method for solving the $u$-subproblem (\ref{ADMM_u}) is presented in Algorithm \ref{AbstractCG}. Compared with the classical CG method (see e.g., Chapter 3 of \cite{glowinski2003finite} and
Chapter 2 of \cite{glowinski2015variational}), Algorithm \ref{AbstractCG} requires updating the adjoint variable $p$ to verify the specific inexactness criterion (\ref{inexact_criterion}). It is clear that the update of $p_{m+1}^k$ can be computed cheaply. Hence, our proposed inexactness criterion (\ref{inexact_criterion}) can be verified by negligible extra computation. More discussions, including the convergence properties of CG type methods applied to the solution of linear systems in Hilbert spaces, can also be found in the mentioned references.

Now, with these discussions, Algorithm \ref{InADMM_algorithm} can be specified as an ADMM--CG two-layer nested iterative scheme for the problem (\ref{Basic_Problem})--(\ref{state_eqn}). We list it as Algorithm \ref{ADMM_CG}.
\begin{algorithm}[h]
	\caption{{An ADMM--CG two-layer nested iterative scheme for the problem (\ref{Basic_Problem})--(\ref{state_eqn}).}}
	\begin{algorithmic}[0]
		\Require  $\{u^0,z^0,\lambda^0\}^\top$ in $U\times U\times U$, $\beta>0$ and $0<\sigma<\frac{\sqrt{2}}{\sqrt{2}+\sqrt{\beta}}\in(0,1).$
		\For {$k\geq 0$} $\{u^k,z^k,\lambda^k\}\rightarrow u^{k+1}\rightarrow z^{k+1}\rightarrow \lambda^{k+1}$ via
		\State Compute $u^{k+1}$ by the CG method in Algorithm \ref{AbstractCG};
		\State Compute $z^{k+1}$ by (\ref{soln_z});
		\State Update the Lagrange multiplier $\lambda^{k+1} = \lambda^k-\beta(u^{k+1}-z^{k+1}).$
		\EndFor
	\end{algorithmic}\label{ADMM_CG}
\end{algorithm}

\begin{remark}\label{re:CG_replaced}
As mentioned, to execute the inexactness criterion (\ref{inexact_criterion}), the CG method can be replaced by other numerical schemes such as the preconditioned MinRes method in \cite{pearson2012regularization} which has been verified to be efficient for unconstrained parabolic optimal control problems. Hence, depending on how to satisfy the inexactness criterion (\ref{inexact_criterion}) internally, Algorithm \ref{InADMM_algorithm} can be specified as various algorithms.
\end{remark}

\section{Numerical Results of Algorithm \ref{ADMM_CG} for (\ref{Basic_Problem})--(\ref{state_eqn})}\label{sec:numerical}

In this section, we report some preliminary numerical results to validate the efficiency of Algorithm \ref{ADMM_CG} for the parabolic optimal control problem (\ref{Basic_Problem})--(\ref{state_eqn}). All codes were written in MATLAB R2016b and numerical experiments were conducted on a Surface Pro 5 laptop with 64-bit Windows 10.0 operation system, Intel(R) Core(TM) i7-7660U CPU (2.50 GHz), and 16 GB RAM.

First, for numerical discretization, we employ the backward Euler finite difference method (with step size $\tau$) for the time discretization and piecewise linear finite element method (with mesh size $h$) for the space discretization. In order to implement (\ref{soln_z}), we perform at each time step a \emph{nodal projection} of the continuous piecewise affine function $\left(u_h^{k+1}-\frac{\lambda_h^k}{\beta}\right)(n\tau)(\in V_{0h})$ over the convex set $\mathcal{C}\cap V_{0h}$, where
$
\mathcal{C}=\{\phi|\phi\in L^2(\Omega),a\leq\phi\leq b\},
$
and (assuming that $\Omega$ is a bounded polygonal domain of $\mathbb{R}^2$)
$$
V_{0h}=\{\phi|\phi\in C^{0}(\bar{\Omega}),\phi|_\mathbb{T}\in P_1, \forall \mathbb{T}\in\mathcal{T}_h, \phi|_\Gamma=0\}.
$$
Here, $\mathcal{T}_h$ is a triangulation of $\Omega$ and $P_1$ is the space of the polynomial functions of two
variables of degree $\leq 1$. In addition, we denote by $P^{\text{nodal}}_{\mathcal{C}\cap V_{0h}}$ the above projection operator, which is defined by
\begin{flalign}\label{r7}
\left\{
\begin{aligned}
&P^{\text{nodal}}_{\mathcal{C}\cap V_{0h}}(\phi)\in \mathcal{C}\cap V_{0h}, \forall \phi\in V_{0h},\\
&P^{\text{nodal}}_{\mathcal{C}\cap V_{0h}}(\phi)(Q_k)=\max\{a,\min\{b,\phi(Q_k)\}\},\forall k=1,\cdots, N_{0h}.
\end{aligned}
\right.
\end{flalign}
In (\ref{r7}), $\{Q_k\}_{k=1}^{N_{0h}}$ is the set of the vertices of triangulation $\mathcal{T}_h$ not located on $\Gamma$. This nodal projection can facilitate the implementation of Algorithm \ref{ADMM_CG}; and we refer to Remark 5 in \cite{GSY2019} for more discussions.

For the linear systems arising at each time step of the discretized parabolic equations, they are solved by the permuted LDL factorization in, e.g., \cite{S96}, because the coefficient matrices are sparse and invariant. Other methodologies such as Krylov subspace methods, domain decomposition methods and multi-grid methods can also be applied to further improve the numerical efficiency. In addition, an adjoint approach is employed for the $u$-subproblem (\ref{ADMM_u}), which requires storing the solution of the state equation (\ref{state_eqn}) at each time step. This is a demanding request on memory, and it may not be applicable for, e.g., time-dependent problems in three-dimensional space, due to the huge scale of systems after discretization. To tackle this issue, some memory saving methodologies can be embedded into our algorithmic design. All these numerical techniques are important but beyond the scope of our discussion; we refer to \cite{saad2003iterative} for fast linear algebra solvers and \cite{berggren1996computational} for a memory saving strategy.

To test the efficiency of Algorithm 3, the primal residual $\pi_s$ and dual residual $d_s$ are respectively defined as
\begin{eqnarray*}
	\pi_s=\|{z}^{k}-{z}^{k-1}\|_{L^2(\mathcal{O})}/\|z^{k-1}\|_{L^2(\mathcal{O})},\quad &d_s=\|{u}^{k}-{z}^{k}\|_{L^2(\mathcal{O})}/\max\{\|u^{k-1}\|_{L^2(\mathcal{O})},\|z^{k-1}\|_{L^2(\mathcal{O})}\}.
\end{eqnarray*}
The stopping criterion for all numerical experiments is
$$
\max\{\pi_s, d_s\}\leq tol,
$$
where $tol>0$ is a prescribed tolerance. The initial values are set as $u=0,z=0$ and $\lambda=0$ in the following discussion. For the constant $\sigma$ in the inexactness criterion (\ref{inexact_criterion}), according to (\ref{sigma}), we choose $\sigma=0.99 \frac{\sqrt{2}}{\sqrt{2}+\sqrt{\beta}}$ because larger values of $\sigma$ mean that the criterion is looser and hence less computation is needed for solving the subproblems.
In addition, we define the relative distance $\hbox{``RelDis''}$ and the objective functional value $\hbox{``Obj''}$ as
\begin{equation*}
\text{RelDis}:={\|y-y_d\|^2_{L^2(Q)}}/{\|y_d\|^2_{L^2(Q)}} \quad\text{and}\quad \text{Obj}:=\frac{1}{2}\|y-y_d\|^2_{L^2(Q)}+\frac{\alpha}{2}\|u\|_{L^2(\mathcal{O})}^2,
\end{equation*}
to verify the accuracy of the numerical solution.

\medskip
\noindent\textbf{Example 1.} We consider an example of the problem (\ref{Basic_Problem})--(\ref{state_eqn}) with a known exact solution; it is a variant of the problem discussed in \cite{andrade2012multigrid}. The model is
\begin{equation}\label{ex1_Problem}
\begin{aligned}
\min_{u\in {\mathcal{C}}, y\in L^2(Q)} \quad  &\frac{1}{2}\iint_Q |y-y_d|^2dxdt+\frac{\alpha}{2}\iint_{Q}|u|^2dxdt\\
{\hbox{s.t.}}\qquad &
\left\{
\begin{aligned}
&\frac{\partial y}{\partial t}-\Delta y=f+u,&\quad \text{in}\quad \Omega\times(0,T), \\
&y=0,&\quad \text{on}\quad \Gamma\times(0,T),\\
& y(0)=\varphi,&
\end{aligned}
\right.
\end{aligned}
\end{equation}
with $\Omega=\{(x_1,x_2)\in \mathbb{R}^2|0<x_1<1, 0<x_2<1\}$, $\omega=\Omega$, $T=1$. In (\ref{ex1_Problem}), the function $f\in L^2(Q)$ is a source term that helps us construct the exact solution without affection to the numerical implementation. We further let
\begin{eqnarray*}
	y=(1-t)\sin\pi x_1\sin\pi x_2,~ p=\alpha (1-t)\sin 2\pi x_1\sin 2\pi x_2,~ u=\min(a,\max(b,-\frac{p}{\alpha})),
\end{eqnarray*}
and set
\begin{equation*}
f=-u+\frac{\partial y}{\partial t}-\Delta y, \quad y_d=y+\frac{\partial p}{\partial t}+\Delta p,\quad \varphi=\sin \pi x
_1\sin \pi x_2.
\end{equation*}
Then, it is easy to verify that $(u^*,y^*):=(u, y)$ is the optimal solution of the problem (\ref{ex1_Problem}). Moreover, the admissible set is
\begin{equation*}
{\mathcal{C}}=\{v|v\in L^\infty(\mathcal{O}), -0.5\leq v(x_1,x_2; t)\leq 0.5 ~\text{a.e. in} \, \mathcal{O}\}\subset{L^2(\mathcal{O})}.
\end{equation*}
We set the regularization parameter $\alpha=10^{-5}$ throughout.

We first test Algorithm \ref{ADMM_CG} with different values of $\beta$ to show how its performance depends on the choice of $\beta$. As discussed in Section \ref{sec:InexactADMM}, $\beta$ should be close to 1 to balance the minimization of $\tilde{J}(u)$ and the satisfaction of the control constraint $u\in \mathcal{C}$. On the other hand, it is clear that the system (\ref{spd_u}) becomes increasingly ill-conditioned as $\beta$ decreases; and a smaller $\beta$ tends to result in slower convergence for the CG method. As a result, the trade-off between the inexactness criterion (\ref{inexact_criterion}) and the conditioning of the $u$-subproblem (\ref{ADMM_u}) should also be considered for choosing $\beta$. The results with $\tau=h=2^{-6}$ and different values of $\beta$ are reported in Table \ref{Tab_beta}, in which the notation ``$\text{ADMM}_{Iter}$'' represents the total out-layer ADMM iteration numbers, ``Mean/Max CG'' denote the average and maximum steps of the inner CG method, respectively. Results in Table \ref{Tab_beta} empirically show that $\beta=2$ or $\beta=3$ is a good choice. In the following, we choose $\beta=3$.

\begin{table}[ht]
	\setlength{\abovecaptionskip}{0pt}
	\setlength{\belowcaptionskip}{5pt}
	\centering
	\caption{Numerical results of Algorithm \ref{ADMM_CG} with different $\beta$ for Example 1.}\label{Tab_beta}
	{\small\begin{tabular}{|c|c|c|c|c|c|c|c|}
			\hline
			$\beta$&0.1&0.5&1&2&3&4&5\\
			\hline
			$\text{ADMM}_{Iter}$&297&60& 29&20 &22&25&29\\
			\hline
			Mean/Max CG& 6.01/10 &7.80/10& 7.48/10& 6.75/9&6.00/8&5.36/7&4.97/7\\
			\hline
		\end{tabular}
	}
\end{table}

Next, we validate the efficiency of the inexactness criterion (\ref{inexact_criterion}). We compare Algorithm \ref{ADMM_CG} with the intuitive implementation of the ADMM (\ref{ADMM}) whose accuracy for solving the $u$-subproblem (\ref{ADMM_u}) by the CG method is empirically set as a constant a prior. For this set of numerical experiments, $tol=10^{-4}$ and various space mesh sizes $h$ and time step sizes $\tau$ as $h=\tau={2^{-i}}$ with $i=5, 6, 7, 8$, are considered. The accuracy for solving the $u$-subproblem (\ref{ADMM_u}) is $ \|e_k(u_{m+1}^k)\|\le 10^{-j}$ with $j$ an integer. We test various values for the accuracy constant: $j=2,4, 6, 8, 10$, which represent from low to very high levels of accuracy. Numerical results are reported in Table \ref{Table_ex1}, in which ``$\hbox{ADMM}_{1e-j}$'' denotes the accuracy constant for solving the $u$-subproblem (\ref{ADMM_u}) is $10^{-j}$. Here and in what follows, the notation ``$\sim$'' means that the ADMM does not converge within 500 iterations.
\begin{table}[ht]
	\setlength{\abovecaptionskip}{0pt}
	\setlength{\belowcaptionskip}{5pt}
	\centering
	\caption{Numerical comparison of Algorithm \ref{ADMM_CG} and $\hbox{ADMM}_{1e-k}$ for Example 1.}\label{Table_ex1}
	{\footnotesize\begin{tabular}{|c|c|c|c|c|c|c|}
			\hline
			{ Mesh}&{ Algorithm}&${\text{ADMM}_{Iter}}$&{Mean/Max CG}&{ Time (s)}&{ RelDis}&{ Obj}\\
			\hline
			&$\hbox{ADMM}_{1e-10}$ &21&61.71/83&17.49&$7.5987\times10^{-7}$&$3.6825\times10^{-7}$\\
			\cline{2-7}
			&$\hbox{ADMM}_{1e-8}$
			&21&44.81/65&16.94&$7.5987\times10^{-7}$&$3.6825\times10^{-7}$\\
			\cline{2-7}				
			$2^{-5}$&$\hbox{ADMM}_{1e-6}$&21&28.47/49&8.59&$7.5986\times10^{-7}$&$3.6825\times10^{-7}$\\
			\cline{2-7}	
			&$\hbox{ADMM}_{1e-4}$ &21&13.30/32&4.23&$7.5990\times10^{-7}$&$3.6825\times10^{-7}$\\
			\cline{2-7}
			&$\hbox{ADMM}_{1e-2}$ &$\sim$&$\sim$&$\sim$&$\sim$&$\sim$\\
			\cline{2-7}
			&{\bf{Algorithm 3}} &24&5.88/8& 1.93&$7.5954\times10^{-7}$&$3.6823\times10^{-7}$\\
			\hline
			
			&$\hbox{ADMM}_{1e-10}$ &19&60.20/94&196.68&$6.7055\times10^{-7}$&$3.5036\times10^{-7}$\\
			\cline{2-7}	   &$\hbox{ADMM}_{1e-8}$&19&45.05/71&170.48&$6.7055\times10^{-7}$&$3.5036\times10^{-7}$\\
			\cline{2-7}	$2^{-6}$&$\hbox{ADMM}_{1e-6}$&19&27.47/48&93.65&$6.7055\times10^{-7}$&$3.5036\times10^{-7}$\\
			\cline{2-7}
			&$\hbox{ADMM}_{1e-4}$ &19&12.84/31&46.79&$6.7056\times10^{-7}$&$3.5035\times10^{-7}$\\
			\cline{2-7}
			&$\hbox{ADMM}_{1e-2}$ &$\sim$&$\sim$&$\sim$&$\sim$&$\sim$\\
			\cline{2-7}
			&{\bf{Algorithm 3}} &22&6.00/8&20.86&$6.7075\times10^{-7}$&$3.5035\times10^{-7}$\\
			
			\hline
			
			&$\hbox{ADMM}_{1e-10}$ &19&59.10/93&3372.30&$6.5295\times10^{-7}$&$3.4473\times10^{-7}$\\
			\cline{2-7}
			&$\hbox{ADMM}_{1e-8}$ &19&44.25/70&2884.61&$6.5295\times10^{-7}$&$3.4473\times10^{-7}$\\
			\cline{2-7}
			$2^{-7}$&$\hbox{ADMM}_{1e-6}$&19&27.15/48&1653.10&$6.5295\times10^{-7}$&$3.4473\times10^{-7}$\\
			\cline{2-7}
			&$\hbox{ADMM}_{1e-4}$ &19&12.70/30&793.48&$6.5299\times10^{-7}$&$3.4473\times10^{-7}$\\
			\cline{2-7}
			&$\hbox{ADMM}_{1e-2}$&$\sim$&$\sim$&$\sim$&$\sim$&$\sim$\\
			\cline{2-7}
			&{\bf{Algorithm 3}} &20&6.20/8&307.06&$6.5294\times10^{-7}$&$3.4473\times10^{-7}$\\
			\hline
			
			&$\hbox{ADMM}_{1e-10}$ &19&58.30/93&37106.76&$6.4876\times10^{-7}$&$3.4260\times10^{-7}$\\
			\cline{2-7}
			&$\hbox{ADMM}_{1e-8}$ &19&43.45/70&26570.61&$6.4876\times10^{-7}$&$3.4260\times10^{-7}$\\
			\cline{2-7}			
			$2^{-8}$&$\hbox{ADMM}_{1e-6}$&19&26.95/48&15801.55&$6.4876\times10^{-7}$&$3.4260\times10^{-7}$\\
			\cline{2-7}
			&$\hbox{ADMM}_{1e-4}$ &19&12.55/30&7627.94&$6.4879\times10^{-7}$&$3.4260\times10^{-7}$\\
			\cline{2-7}
			&$\hbox{ADMM}_{1e-2}$ & $\sim$&$\sim$&$\sim$&$\sim$&$\sim$\\
			\cline{2-7}
			&{\bf{Algorithm 3}} &20&6.05/8&3839.67&$6.4877\times10^{-7}$&$3.4260\times10^{-7}$\\
			\hline
		\end{tabular}
	}
\end{table}

According to Table \ref{Table_ex1}, the automatically adjustable inexactness criterion (\ref{inexact_criterion}) is favorable for the implementation of ADMM (\ref{ADMM}). If the accuracy is set as a constant a prior, then it is not easy to probe an appropriate value. An either too large or too small value may result in troubles. For a too large value, e.g., $10^{-2}$, the accuracy for solving the subproblems may not be sufficient and the convergence may not be guaranteed. For a too small value, e.g., $10^{-8}$ or $10^{-10}$, the accuracy for solving the subproblems may be unnecessarily high and it does not help accelerate the overall convergence. Especially, if the mesh size for discretization is small, then the resulting $u$-subproblem is high dimensional and it becomes less practical to solve it to a high precision. For the cases tested, retrospectively, the accuracy $10^{-4}$ is a good choice. But there is neither theory nor hint to fathom this value a prior. Indeed, as to be shown in Example 2, this value could be heavily dependent on the specific problem under discussion. The inexactness criterion (\ref{inexact_criterion}), however, can find an appropriate accuracy automatically for finding an approximate solution of the $u$-subproblem (\ref{ADMM_u}). Hence, Algorithm \ref{ADMM_CG} does not have these mentioned difficulties, and it generally works well for all the tested cases.
{Table \ref{Table_ex1} also shows that the efficiency of Algorithm \ref{ADMM_CG} is independent from the mesh size used for discretization. This is an important feature to guarantee the numerical efficiency when an algorithm is applied to the discretized version of some model with fine mesh for discretization, as mentioned in some well-known works such as \cite{BKI99,HKV09,hinze2008optimization,KKV2011}.
	
\begin{table}[ht]
	\setlength{\abovecaptionskip}{0pt}
	\setlength{\belowcaptionskip}{5pt}
	\centering
	\caption{Numerical errors of Algorithm \ref{ADMM_CG} with $\beta=3$ and $tol=10^{-4}$ for Example 1.}\label{err_ex1}
	{\small\begin{tabular}{|c|c|c|c|c|c|}
			\hline
			error&$h=\tau=2^{-5}$&$h=\tau=2^{-6}$&$h=\tau=2^{-7}$&$h=\tau=2^{-8}$\\
			\hline
			$\|u-u^*\|_{L^2(\mathcal{O})}$&$1.8421\times 10^{-2}$ & $4.6767\times 10^{-3}$&$1.1715\times 10^{-3}$ &$2.9013\times 10^{-4}$\\
			\hline
			$\|y-y^*\|_{L^2(Q)}$& $ 3.6426\times 10^{-5}$ & $  8.6088\times 10^{-6}$& $ 2.1106\times 10^{-6}$&$4.9269\times 10^{-7}$\\
			\hline
		\end{tabular}
	}
\end{table}

Since the ADMM (\ref{ADMM}) is a first-order algorithm and generally it is not favorable to generate iterates in very high precisions, it is necessary to verify if the ADMM (\ref{ADMM}) can be accurate enough to guarantee the iterative accuracy. In other words, whether or not it is still the discretization error that constitutes the main part of the total error when the ADMM (\ref{ADMM}) is applied to the discretized version of the problem (\ref{ex1_Problem}). Recall that the solution of Example 1 is known. In Table \ref{err_ex1}, we report the $L^2$-error for the iterate ($u$, $y$) obtained by Algorithm \ref{ADMM_CG} for various values of $h$ and $\tau$. For succinctness, we only give the results for the case where $\beta=3$ and $tol= 10^{-4}$. It is clear from Table \ref{err_ex1} that, when the ADMM (\ref{ADMM}) is applied to the problem (\ref{ex1_Problem}), the iterative accuracy is sufficient and the overall error of $u$ and $y$ are both dominated by the discretization error.}

Evolutions of the residuals and objective functional values with respect to the outer ADMM iterations are displayed in Figure \ref{iterative_result_ex1}. These curves indicate the fast convergence of Algorithm \ref{ADMM_CG}. In addition, the state variable $y$ and the control variable $u$, and the errors $y^*-y$ and $u^*-u$ at $t=0.25$  with $h=\tau=2^{-6}$ are depicted in Figures \ref{numerical_result_ex1} and \ref{error_ex1}, respectively.
\begin{figure}[ht]
	\setlength{\abovecaptionskip}{0pt}
	\setlength{\belowcaptionskip}{5pt}
	\caption{Residuals (left) and objective functional values (right) with respect to outer ADMM iterations for Example 1.}\label{iterative_result_ex1}
	\centering
	\includegraphics[width=0.48\textwidth]{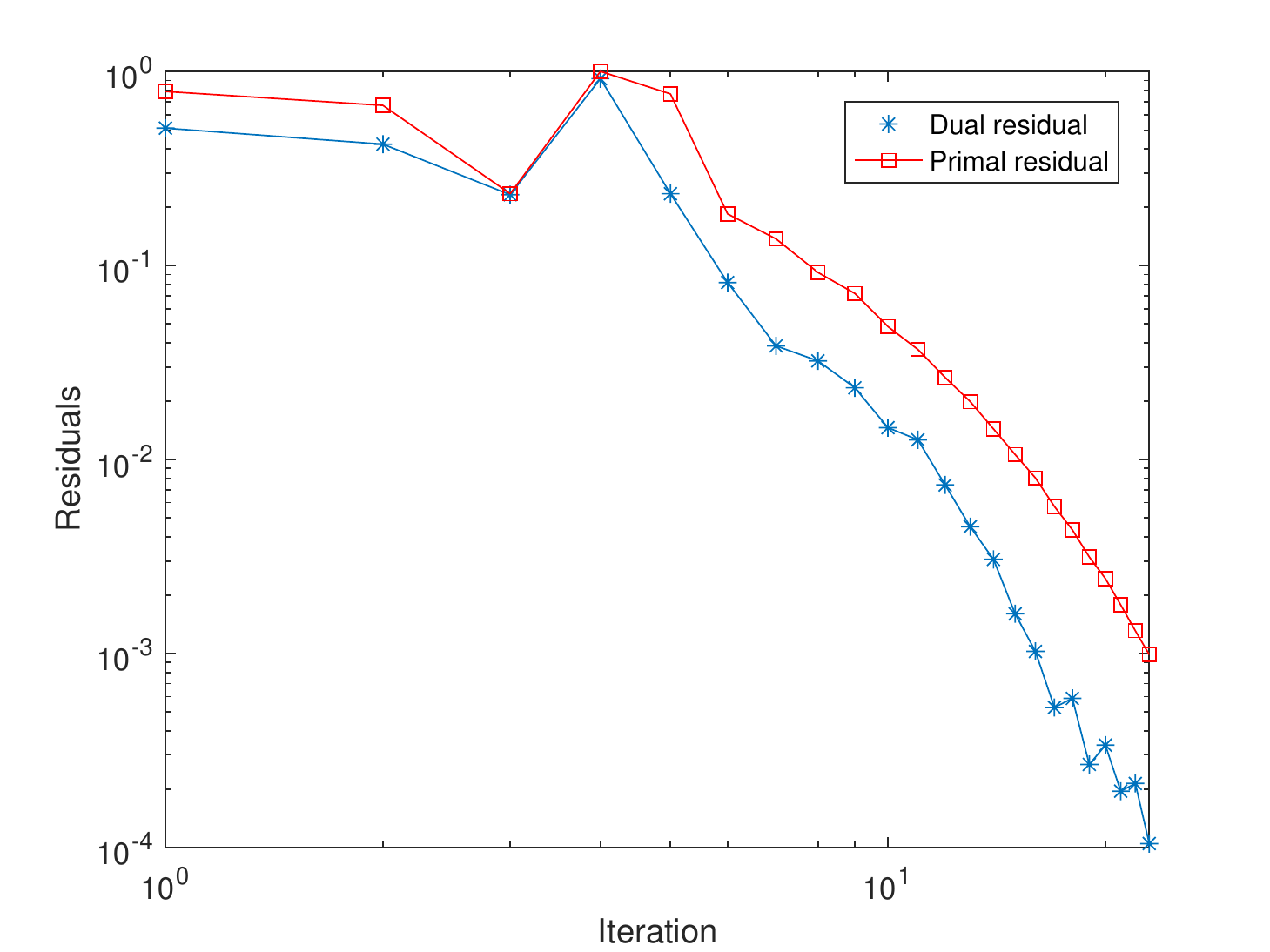}
	\includegraphics[width=0.48\textwidth]{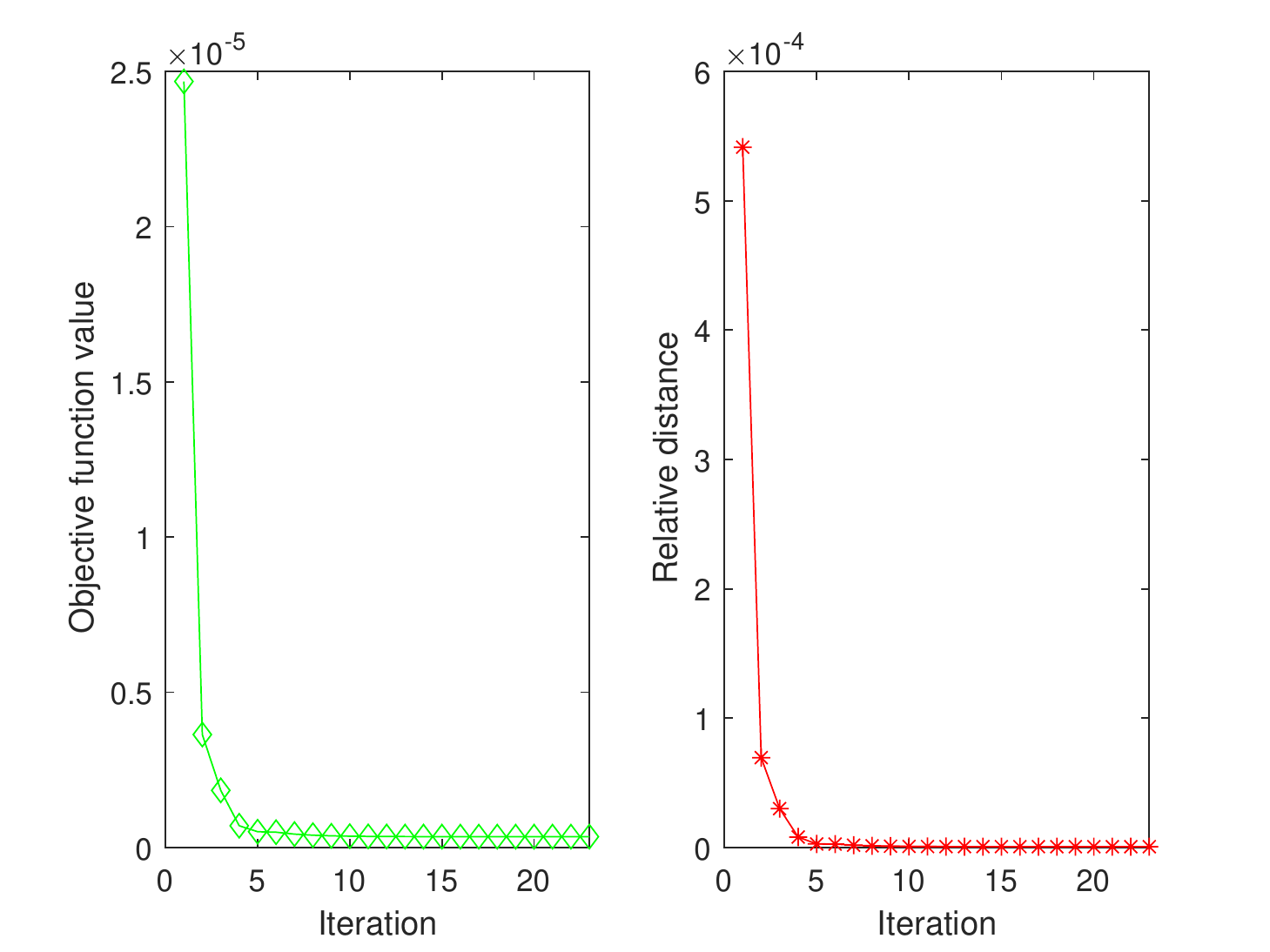}
\end{figure}
\begin{figure}[ht]
	\setlength{\abovecaptionskip}{0pt}
	\setlength{\belowcaptionskip}{5pt}
	\caption{Numerical solutions $y$ (left) and $u$ (right) at $t=0.25$ for Example 1.}\label{numerical_result_ex1}
	\centering
	\includegraphics[width=0.48\textwidth]{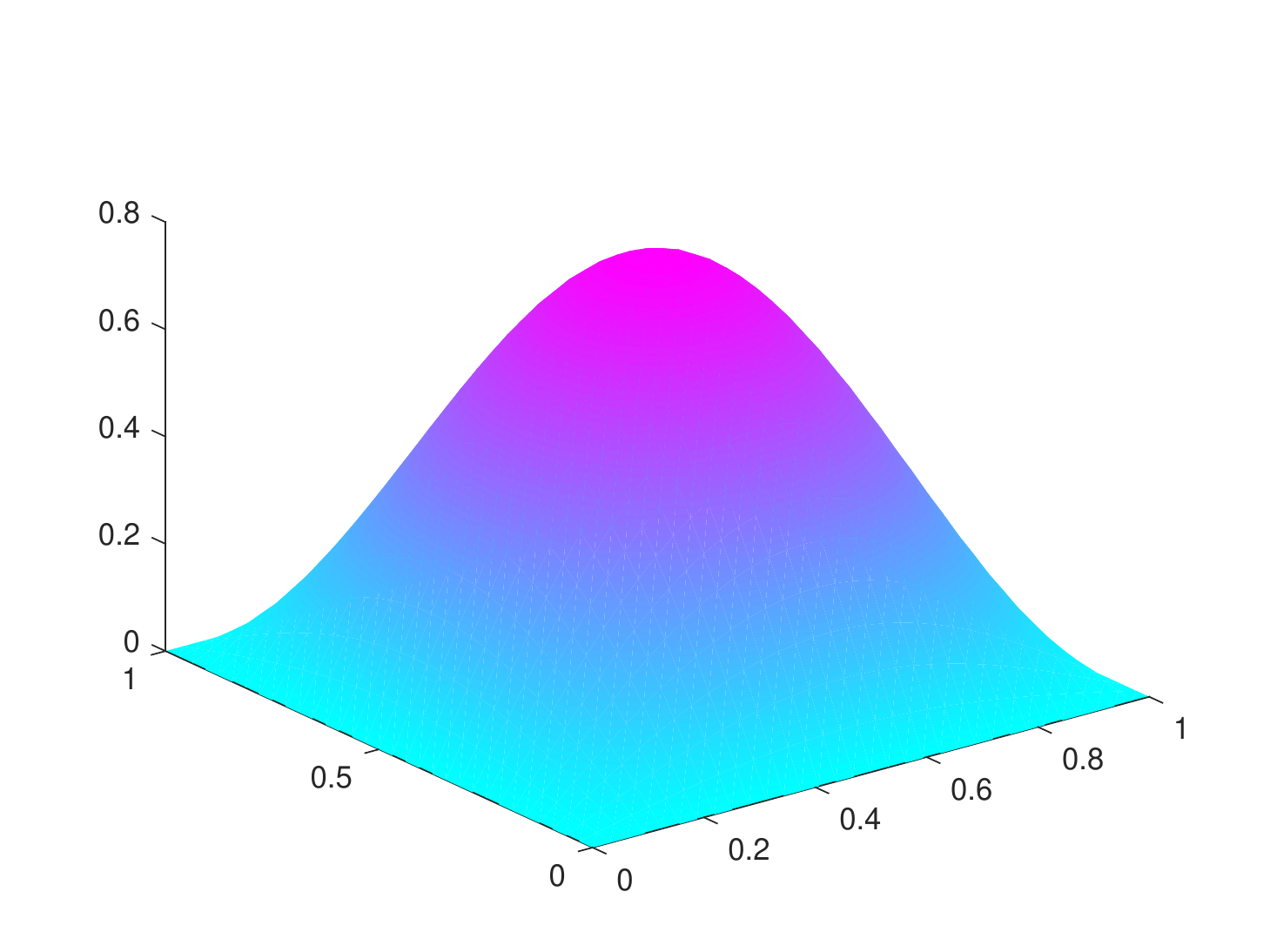}
	\includegraphics[width=0.48\textwidth]{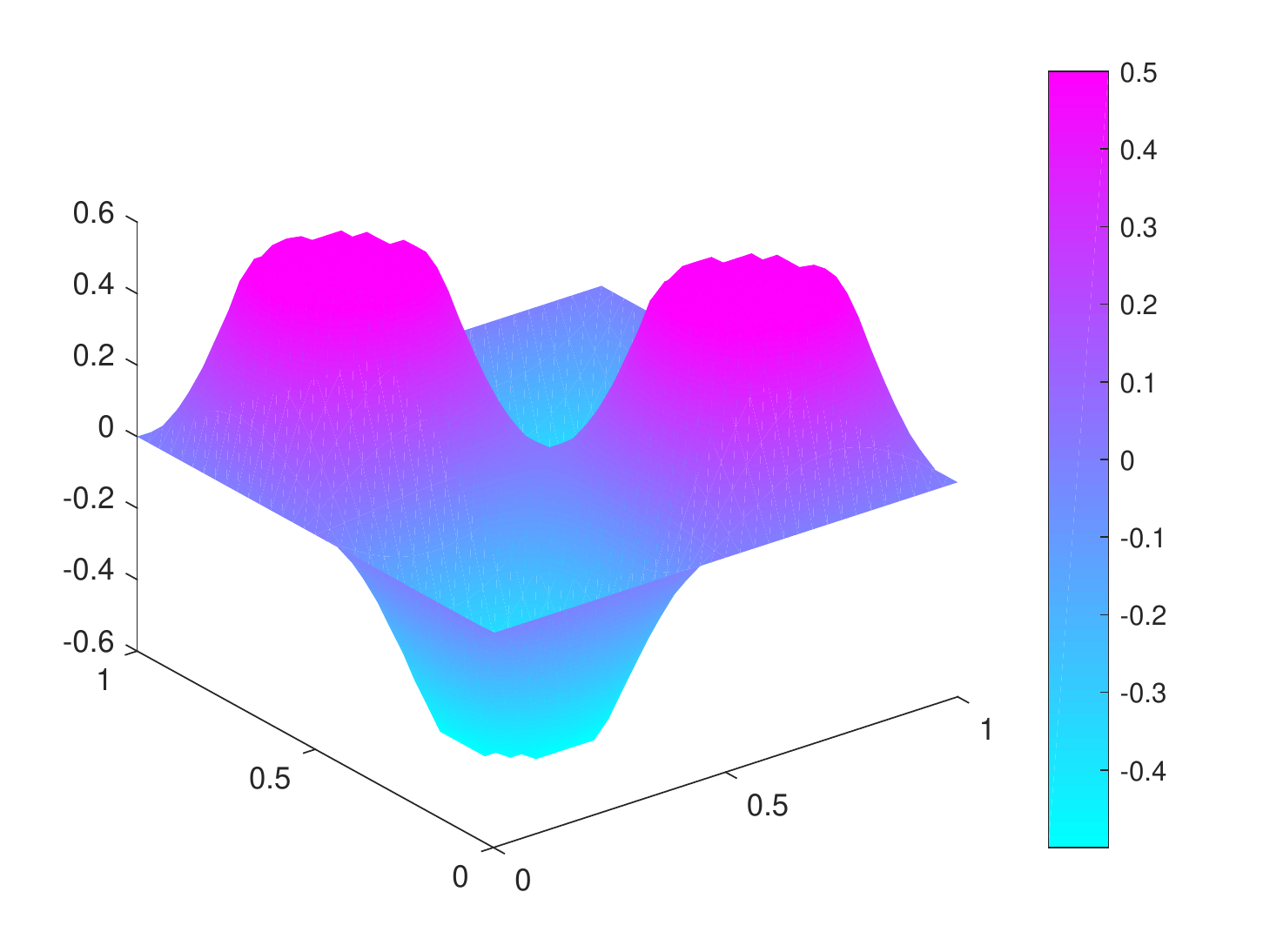}
\end{figure}
\begin{figure}[ht]
	\setlength{\abovecaptionskip}{0pt}
	\setlength{\belowcaptionskip}{5pt}
	\caption{Errors $y^*-y$ (left) and $u^*-u$ (right) at $t=0.25$ for Example 1.}\label{error_ex1}
	\centering
	\includegraphics[width=0.48\textwidth]{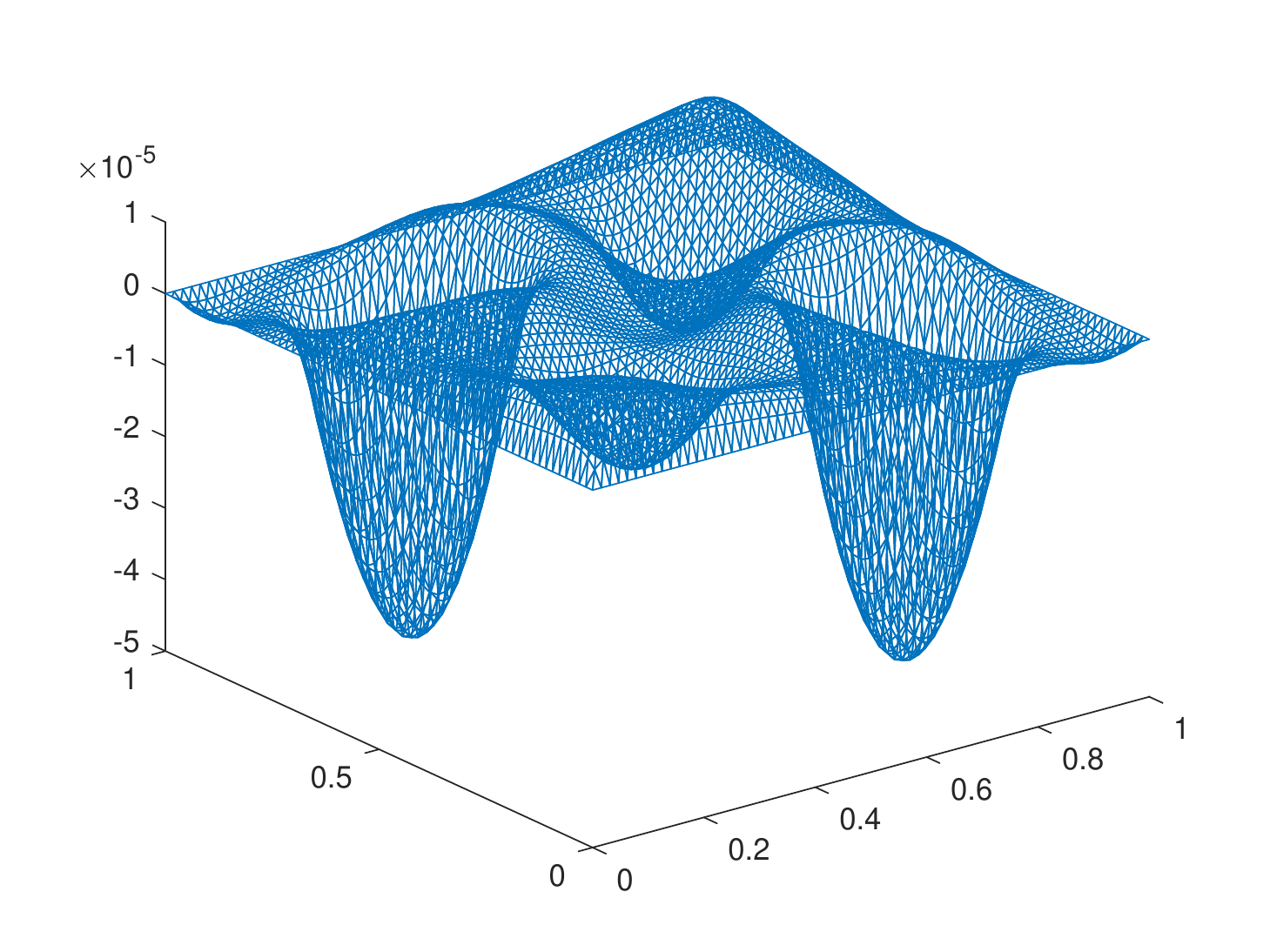}
	\includegraphics[width=0.48\textwidth]{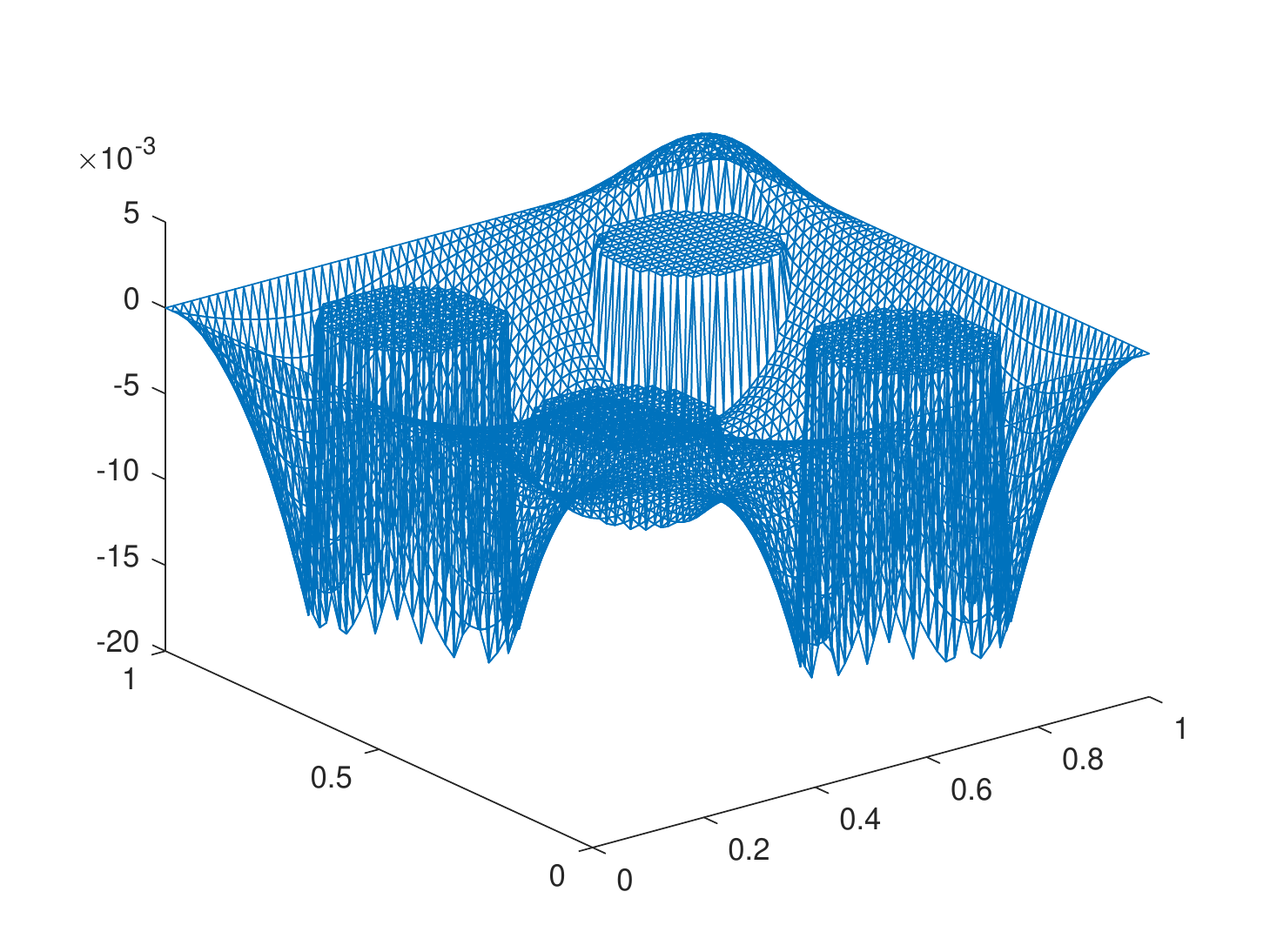}
\end{figure}

\medskip
\noindent\textbf{Example 2.}
We consider another case of the problem (\ref{Basic_Problem})--(\ref{state_eqn}) where the control region $\omega$ is a subset of the domain $\Omega$. Let $\Omega=\{(x_1,x_2)\in \mathbb{R}^2|0<x_1<1, 0<x_2<1\}$, $\omega=\{(x_1,x_2)\in \mathbb{R}^2|0<x_1<0.25, 0<x_2<0.25\}\subsetneq \Omega$ and $Q=\Omega\times(0,T), \mathcal{O}=\omega\times(0,T)$ with $T=1$. The regularization parameter $\alpha=10^{-6}$ and the admissible set is defined as
\begin{equation*}
{\mathcal{C}}=\{v|v\in L^\infty(\mathcal{O}), -300\leq v(x_1,x_2; t)\leq 300 ~\text{a.e. in} \, \mathcal{O}\}\subset{L^2(\mathcal{O})}.
\end{equation*}
The target function $y_d$ is given by
$
y_d=e^t\sin 4x_1\sin 4x_2,
$
and the coefficients $\nu=a_0=1$.

We set $\beta=3$ and $tol=10^{-3}$ throughout, and test various choices of the mesh size. The numerical results are summarized in Table \ref{Table_ex2}. Residuals and the objective functional values are plotted in Figure \ref{iterative_result_ex2}; numerical results for $y$ and $u$ with $h=\tau=2^{-6}$ at $t=0.5$ are presented in Figure \ref{numerical_result_ex2}. We observe that Algorithm \ref{ADMM_CG} is also very efficient and robust for the small control region case; and solving the $u$-subproblem (\ref{ADMM_u}) subject to the inexactness criterion (\ref{inexact_criterion}) reduces the computational cost significantly. Similar conclusions as those for Example 1 can be drawn for this example.

\begin{table}[ht]
	\setlength{\abovecaptionskip}{0pt}
	\setlength{\belowcaptionskip}{5pt}
	\centering
	\caption{Numerical comparison of Algorithm \ref{ADMM_CG} and $\hbox{ADMM}_{1e-k}$ for Example 2.}\label{Table_ex2}
	{\footnotesize\begin{tabular}{|c|c|c|c|c|c|c|}
			\hline
			Mesh&Algorithm&$\text{ADMM}_{Iter}$&Mean/Max CG&Time (s)&RelDis&Obj \\
			\hline
			
			&$\hbox{ADMM}_{1e-10}$ &14&51.50/62&8.56&0.9388&0.3726\\
			\cline{2-7}
			&$\hbox{ADMM}_{1e-8}$
			&14&41.86/52&6.64&0.9388&0.3726\\
			\cline{2-7}				
			$2^{-5}$&$\hbox{ADMM}_{1e-6}$&14&32.64/43&5.34&0.9388&0.3726\\
			\cline{2-7}	
			&$\hbox{ADMM}_{1e-4}$ &14&23.00/32&3.70&0.9388&0.3726\\
			\cline{2-7}
			&$\hbox{ADMM}_{1e-2}$ &$14$&13.71/23&2.24&0.9388&0.3726\\
			\cline{2-7}
			&{\bf{Algorithm 3}} &17&3.35/4&0.83&0.9388&0.3726\\
			\hline
			
			&$\hbox{ADMM}_{1e-10}$ &16&51.63/62&110.05&0.9428&0.3812\\
			\cline{2-7}	   &$\hbox{ADMM}_{1e-8}$&16&41.88/52&85.20&0.9428&0.3812\\
			\cline{2-7}	$2^{-6}$&$\hbox{ADMM}_{1e-6}$&16&32.31/43&64.43&0.9428&0.3812\\
			\cline{2-7}
			&$\hbox{ADMM}_{1e-4}$ &16&22.81/33&45.82&0.9428&0.3812\\
			\cline{2-7}
			&$\hbox{ADMM}_{1e-2}$ &$16$&13.25/23&27.15&0.9428&0.3812\\
			\cline{2-7}
			&{\bf{Algorithm 3}} &18&3.39/4&9.29&0.9428&0.3812\\
			
			\hline
			
			&$\hbox{ADMM}_{1e-10}$ &16&50.50/61&1834.32&0.9455&0.3821\\
			\cline{2-7}
			&$\hbox{ADMM}_{1e-8}$ &16&41.25/52&1550.68&0.9455&0.3821\\
			\cline{2-7}
			$2^{-7}$&$\hbox{ADMM}_{1e-6}$&16&31.81/42&1291.11&0.9455&0.3821\\
			\cline{2-7}
			&$\hbox{ADMM}_{1e-4}$ &16&22.13/32&883.59&0.9455&0.3821\\
			\cline{2-7}
			&$\hbox{ADMM}_{1e-2}$&$16$&12.81/23&401.55&0.9455&0.3821\\
			\cline{2-7}
			&{\bf{Algorithm 3}} &18&3.33/4&129.33&0.9455&0.3821\\
			\hline
			
			&$\hbox{ADMM}_{1e-10}$ &16&49.69/60&22540.18&0.9470&0.3817\\
			\cline{2-7}
			&$\hbox{ADMM}_{1e-8}$ &16&40.44/51&18869.58&0.9470&0.3817\\
			\cline{2-7}			
			$2^{-8}$&$\hbox{ADMM}_{1e-6}$&16&31.25/41&14969.83&0.9470&0.3817\\
			\cline{2-7}
			&$\hbox{ADMM}_{1e-4}$ &16&22.06/32&10437.38&0.9470&0.3817\\
			\cline{2-7}
			&$\hbox{ADMM}_{1e-2}$ & 16&12.63/22&6281.95&0.9470&0.3817\\
			\cline{2-7}
			&{\bf{Algorithm 3}} &18&3.33/4&1609.73&0.9470&0.3817\\
			\hline
		\end{tabular}
	}
\end{table}

\begin{figure}[ht]
	\setlength{\abovecaptionskip}{0pt}
	\setlength{\belowcaptionskip}{5pt}
	\caption{{Residuals (left) and objective functional values (right) with respect to outer ADMM iterations for Example 2.}}\label{iterative_result_ex2}
	\centering
	\includegraphics[width=0.48\textwidth]{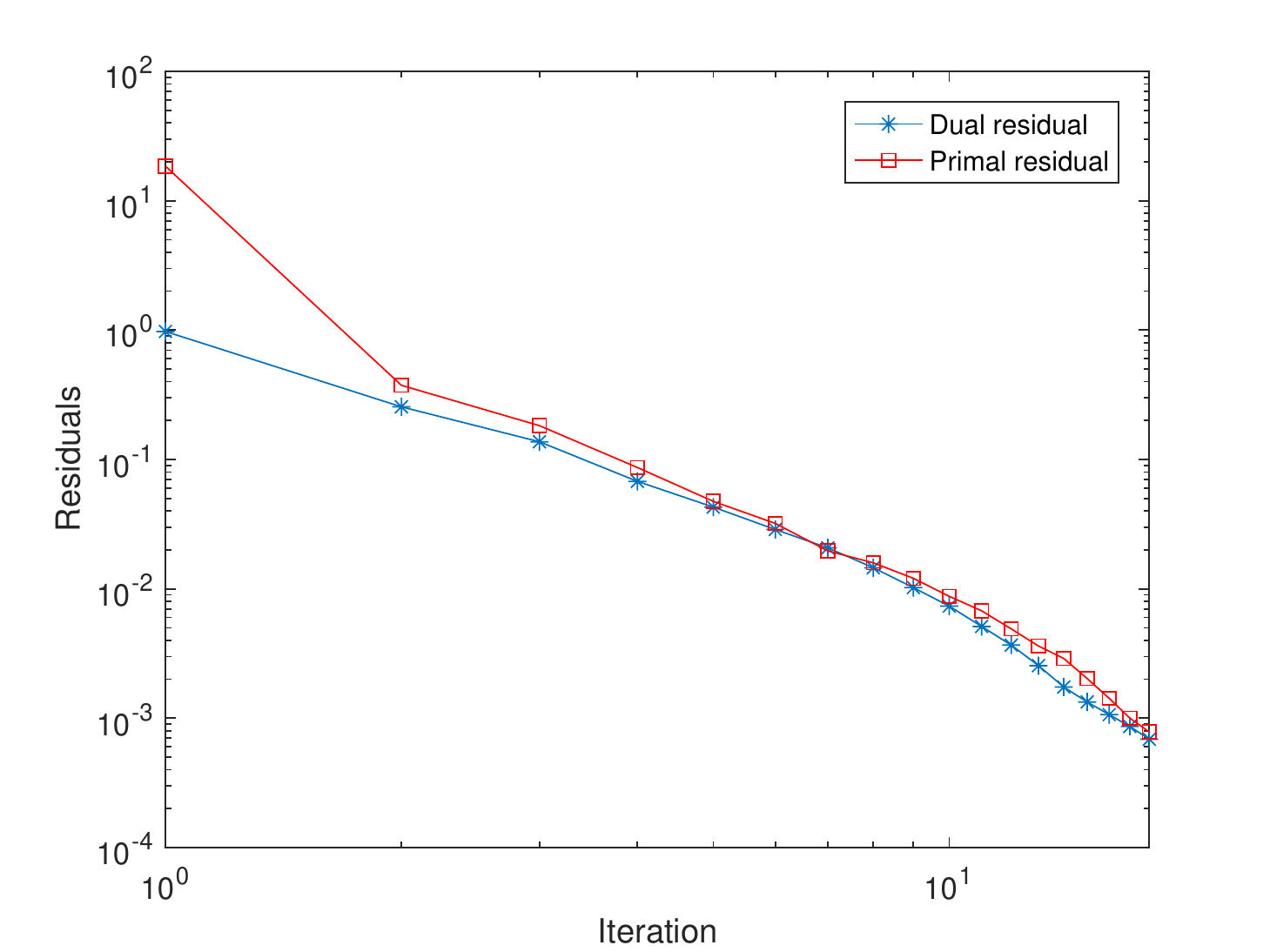}
	\includegraphics[width=0.48\textwidth]{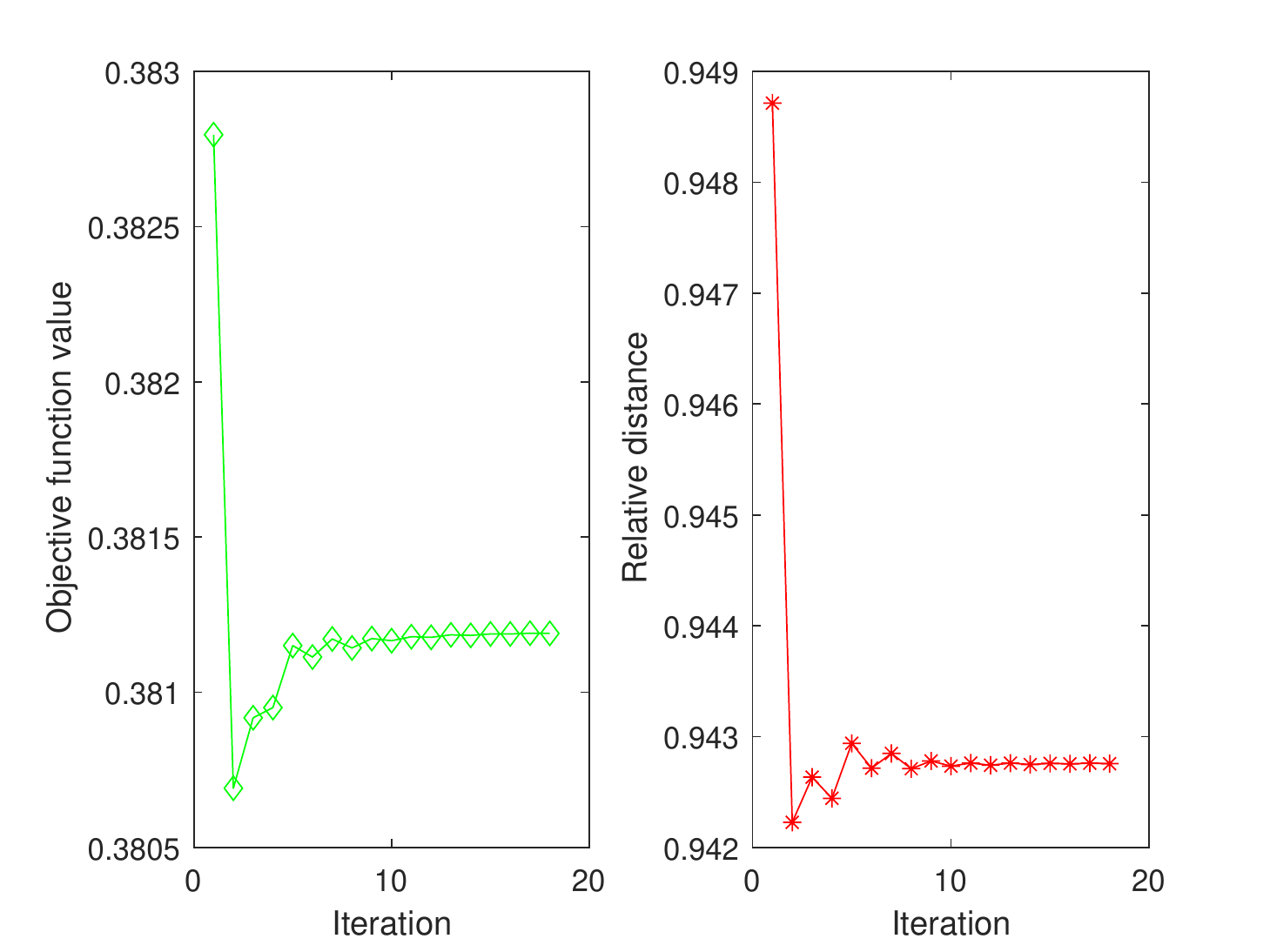}
\end{figure}

\begin{figure}[ht]
	\setlength{\abovecaptionskip}{0pt}
	\setlength{\belowcaptionskip}{5pt}
	\caption{{Numerical solutions $y$ (left) and $u$ (right) at $t=0.5$ for Example 2.}}\label{numerical_result_ex2}
	\centering
	\includegraphics[width=0.48\textwidth]{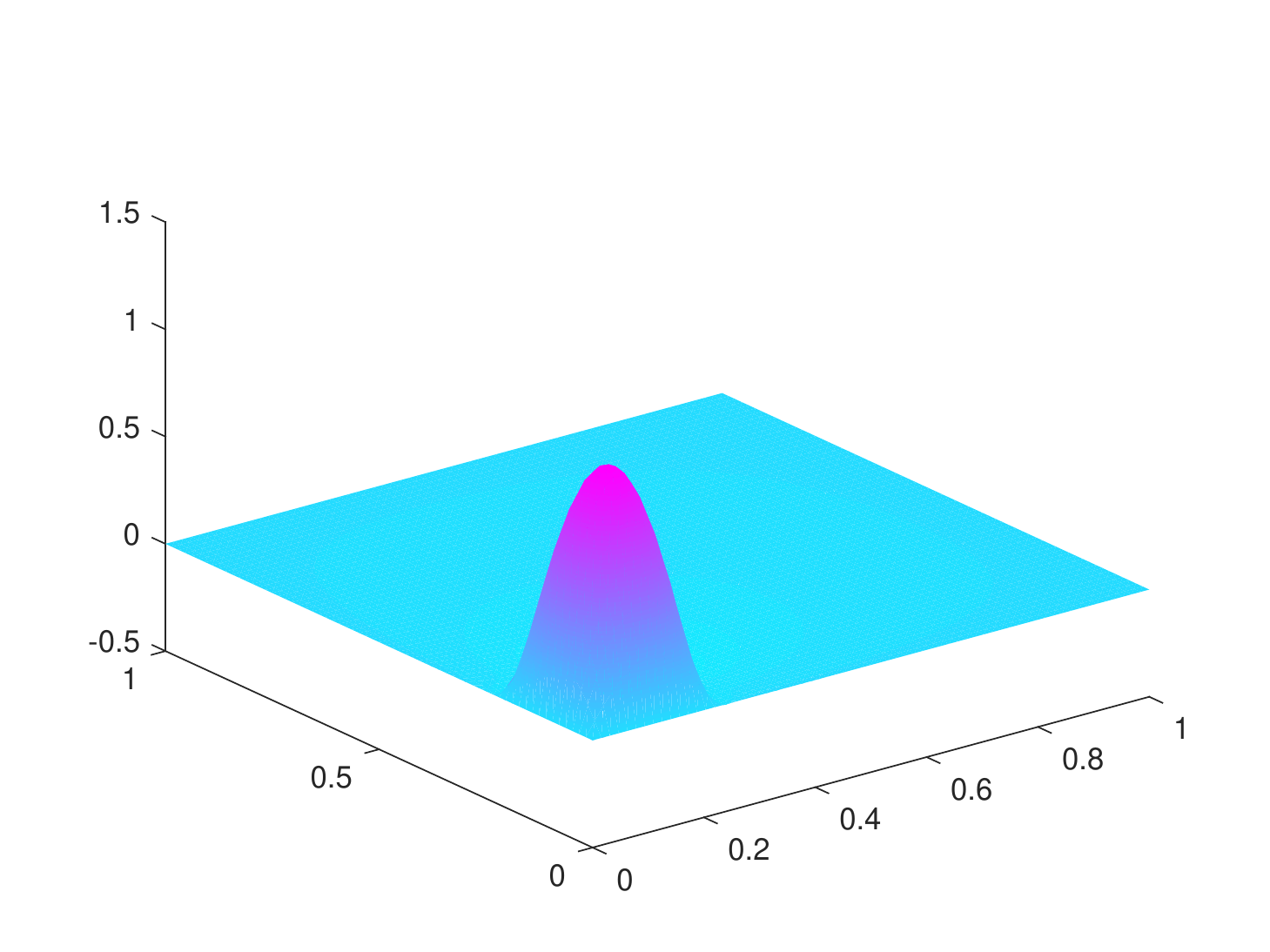}
	\includegraphics[width=0.48\textwidth]{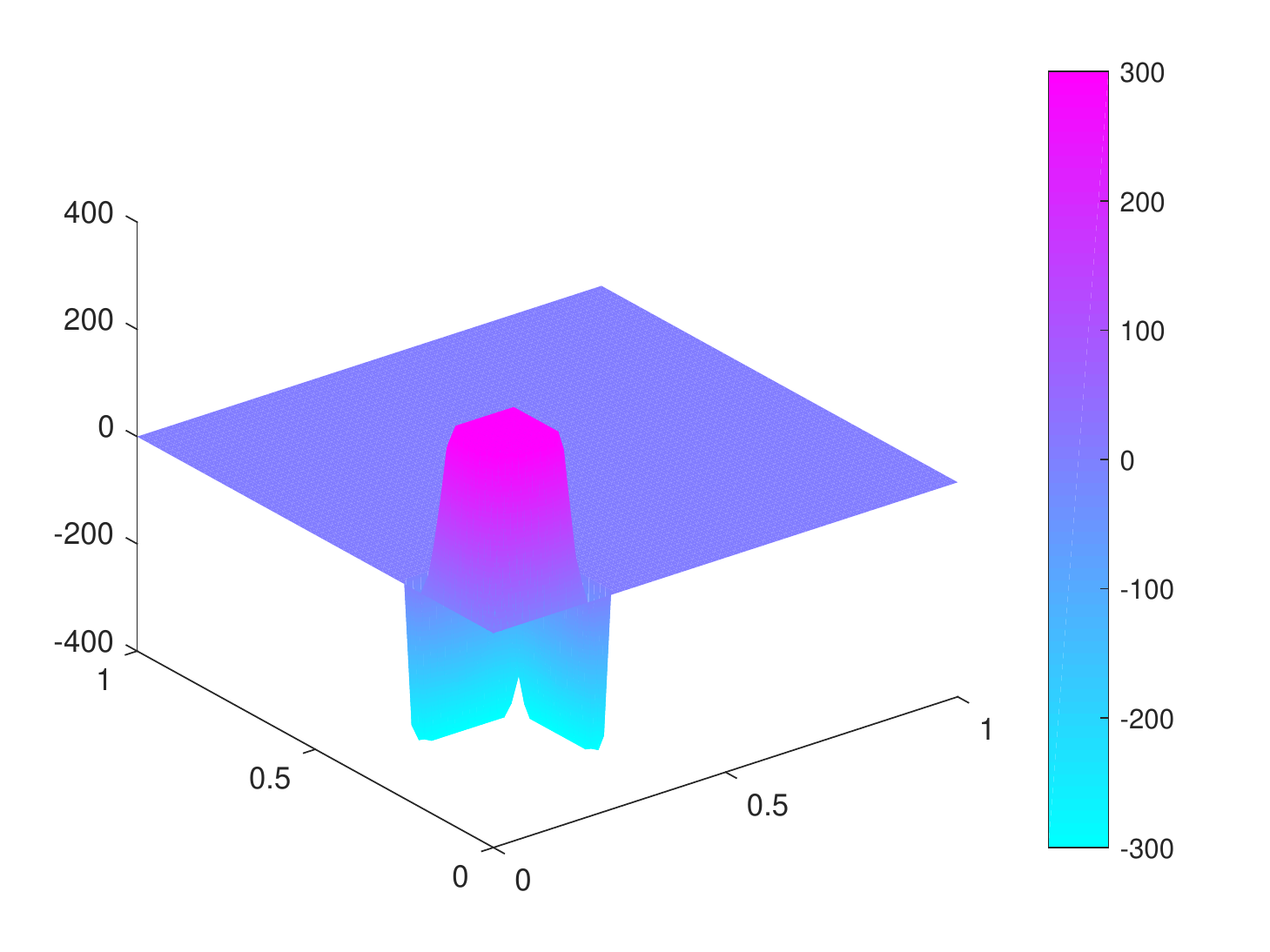}
\end{figure}

\section{Extensions}\label{se:extension}

In previous sections, our discussion is focused on the parabolic optimal control problem with control constraints (\ref{Basic_Problem})--(\ref{state_eqn}) in order to expose our main ideas clearly. The discussion can be easily extended to various other optimal control problems. For instances, the objective functional in (\ref{Basic_Problem}) can be replaced by the $L^1$-control cost functional in \cite{SB2017}, and the control variable $u$ can be replaced by the Neumann or Dirichlet boundary control variable in \cite{glowinski2008exact}. In addition, note that both of the proposed algorithmic design and the theoretical analysis are independent of the specific form of the solution operator $S$ defined in (\ref{def:S}), and they can be extended to various optimal control problems constrained by other linear PDEs. To be more concrete, it is clear that the definition of $e_k(u)$ in (\ref{eu_oper}) is originated from the optimality system of (\ref{ADMM_u}), and it only requires that the solution operator $S$ be affine (i.e., the linearity of the state equation (\ref{state_eqn})). Hence, the parabolic state equation in (\ref{state_eqn}) can be replaced by, e.g., the elliptic equation \cite{hinze12}, the wave equation \cite{glowinski1995exact}, the convection-diffusion equation \cite{glowinski2008exact}, or the fractional parabolic equation \cite{BTY2014}. In this section, we choose two cases to delineate the extensions. Some notations and discussions analogous to previous ones are not repeated for succinctness.

\subsection{Optimal Control Problems Constrained by the Wave Equation}

We first consider the extension to an optimal control problem constrained by the wave equation.

\subsubsection{Model}

We consider the following optimal control problem with control constraints:
\begin{equation}\label{wave_control}
\begin{aligned}
\min_{u\in {\mathcal{C}}, y\in L^2(Q)} \quad  &\frac{1}{2}\iint_Q |y-y_d|^2 dxdt+\frac{\alpha}{2}\iint_{\mathcal{O}}|u|^2 dxdt,\\
\end{aligned}
\end{equation}
and it is subject to the wave equation
\begin{equation}\label{wave_eqn}
\frac{\partial^2 y}{\partial t^2}-\Delta y=u\chi_{\mathcal{O}}~ \text{in}~ \Omega\times(0,T), \quad
y=0~ \text{on}~ \Gamma\times(0,T),\quad
y(0)=y_0,~ \frac{\partial y}{\partial t}(0)=y_1.
\end{equation}
Notation in (\ref{wave_control})--(\ref{wave_eqn}) is the same as that in (\ref{Basic_Problem})--(\ref{state_eqn}) except that the initial conditions $y_0 \in H_0^1(\Omega)$
and $y_1 \in L^2(\Omega)$. For the existence, uniqueness, and regularity of the solution of (\ref{wave_control})--(\ref{wave_eqn}), we refer to, e.g., \cite{lions1971optimal}.

For the special case of (\ref{wave_control})--(\ref{wave_eqn}) where $d=1$ or $\omega=\Omega$, SSN type methods have been studied in the literature, see, e.g., \cite{KKV2011,LLX2016,LP2019}. For the general case of (\ref{wave_control})--(\ref{wave_eqn}) where $d\ge 2$ and $\omega\subsetneq \Omega$, similar difficulties as those mentioned in the introduction for the problem (\ref{Basic_Problem})--(\ref{state_eqn}) arise if SSN type methods are applied. Below, we briefly show the details of extending Algorithm \ref{InADMM_algorithm} to the general case of (\ref{wave_control})--(\ref{wave_eqn}).

\subsubsection{Algorithm}

Similarly, the direct implementation of ADMM to the problem (\ref{wave_control})--(\ref{wave_eqn}) reads as
\begin{subequations}\label{ADMM_s}
	\begin{numcases}
	~u^{k+1}=\arg\min_{u\in L^2(\mathcal{O})}\bar{L}_\beta(u, z^{k},\lambda^k),\label{ADMM_us}\\
	z^{k+1} =\arg\min_{z\in L^2(\mathcal{O})}\bar{L}_\beta(u^{k+1},z,\lambda^k),\label{ADMM_zs}\\
	\lambda^{k+1} = \lambda^k-\beta(u^{k+1}-z^{k+1}),\label{ADMM_lambdas}
	\end{numcases}
\end{subequations}
where the augmented Lagrangian functional $\bar{L}_\beta(u, z,\lambda)$ has the same form as the $L_\beta(u, z,\lambda)$ in (\ref{ADMM}) except that the solution operator $S$ is associated with the wave equation (\ref{wave_eqn}) instead of the parabolic equation (\ref{state_eqn}).

For the $z$-subproblem (\ref{ADMM_zs}), it amounts to computing the projection onto the admissible set $\mathcal{C}$; and the $u$-subproblem (\ref{ADMM_us}) is an unconstrained optimal control problem subject to the wave equation (\ref{wave_eqn}). Note that the $u$-subproblem (\ref{ADMM_us}) shares the same numerical challenges as the subproblem (\ref{ADMM_u}); we may apply the CG method such as \cite{glowinski2008exact} to solve it iteratively at each iteration. To propose the inexactness criterion, we first need to introduce a residual $e_k(u)$ for the $u$-subproblem (\ref{ADMM_us}) as we have done in Section \ref{sec:InexactADMM}. For this purpose, inspired by (\ref{eu_oper}), we define $e_k(u)$ as
\begin{equation*}
e_{k}(u):= (1+\beta) u+S^*(\frac{1}{\alpha}(S(u)-y_d))-\beta z^{k}- \lambda^k,
\end{equation*}
where $S: L^2(\mathcal{O})\longrightarrow L^2(Q)$ is the solution operator associated with the wave equation (\ref{wave_eqn}) and $S^*: L^2(Q)\longrightarrow L^2(\mathcal{O})$ is the adjoint operator of $S$. It is easy to show that
\begin{equation}\label{eu}
e_{k}(u)= (1+\beta) u+p|_{\mathcal{O}}-\beta z^{k}- \lambda^k,
\end{equation}
where $p$ is the successive solution of the wave equation (\ref{wave_eqn}) and the following adjoint equation:
\begin{equation}\label{adjoint_wave}
\frac{\partial^2 p}{\partial t^2}-\Delta p=\frac{1}{\alpha}(y-y_d)~ \text{in}~ \Omega\times(0,T), \quad
p=0~\text{on}~ \Gamma\times(0,T),\quad
p(T)=0,~\frac{\partial p}{\partial t}(T)=0.
\end{equation}
Then, the inexactness criterion for computing $u^{k+1}$ in (\ref{ADMM_us}) is
\begin{equation}\label{inexact_criterion_s}
\|e_{k}(u^{k+1})\|\leq\sigma\|e_{k}(u^{k})\|,
\end{equation}
with the constant $\sigma$ given in (\ref{sigma}).

Although the same letter in (\ref{def_eu}) is used, the definition of $e_k(u)$ in (\ref{eu})
is determined by the wave equation (\ref{wave_eqn}) and the adjoint equation (\ref{adjoint_wave}). It is thus different from (\ref{def_eu}) for the parabolic equation (\ref{state_eqn}) and its adjoint equation (\ref{dis_adjoint}). Embedding the inexactness criterion (\ref{inexact_criterion_s}) into the ADMM scheme (\ref{ADMM_s}), an inexact version of the ADMM (\ref{ADMM_s}) similar as Algorithm 1 is readily available for the problem (\ref{wave_control})--(\ref{wave_eqn}), and its convergence can be proved similarly. We omit the details.

\subsubsection{Numerical Results}

We test the ADMM scheme (\ref{ADMM_s}) with the inexactness criterion (\ref{inexact_criterion_s}), and report some preliminary numerical results for the problem (\ref{wave_control})--(\ref{wave_eqn}) where $\omega\subsetneq \Omega$ and $d=2$.

\medskip
\noindent\textbf{Example 3}. Let us consider the following optimal control problem constrained by the wave equation with a known exact solution:
\begin{equation}\label{sup_ex}
\begin{aligned}
\min_{u\in {\mathcal{C}}, y\in L^2(Q)} \quad  &\frac{1}{2}\iint_Q |y-y_d|^2dxdt+\frac{\alpha}{2}\iint_{\mathcal{O}}|u|^2dxdt\\
{\hbox{s.t.}}\qquad &
\left\{
\begin{aligned}
&\frac{\partial^2 y}{\partial t^2}-\Delta y=f+u\chi_{\mathcal{O}},&\quad \text{in}\quad \Omega\times(0,T), \\
&y=0,&\quad \text{on}\quad \Gamma\times(0,T),\\
& y(0)=y_0,~\frac{\partial y}{\partial t}(0)=y_1,&
\end{aligned}
\right.
\end{aligned}
\end{equation}
where $\Omega=\{(x_1,x_2)\in \mathbb{R}^2|0<x_1<1,  0<x_2<1\}$, $T=1$ and the control region $\omega=\{(x_1,x_2)\in \mathbb{R}^2|0<x_1<0.5, 0<x_2<0.5\}\subsetneq\Omega$.
In addition, we set
\begin{eqnarray*}
	y=e^t\sin\pi x_1\sin\pi x_2,~ p=\sqrt{\alpha} (t-T)^2\sin \pi x_1\sin \pi x_2,~ u=\min(a,\max(b,-\frac{1}{\alpha}{p|_{\mathcal{O}}})),
\end{eqnarray*}
and
\begin{equation*}
f=-u\chi_{\mathcal{O}}+\frac{\partial^2 y}{\partial t^2}-\Delta y, ~ y_d=y-\frac{\partial^2 p}{\partial t^2}+\Delta p,~y_0=\sin \pi x
_1\sin \pi x_2,~ y_1=\sin \pi x
_1\sin \pi x_2.
\end{equation*}
It is easy to verify that $(u^*,y^*):=(u, y)$ is the solution point of the problem (\ref{sup_ex}). Moreover, we set the regularization parameter $\alpha=10^{-4}$ and 
$
{\mathcal{C}}=\{v|v\in L^\infty(\mathcal{O}), -5\leq v(x_1,x_2; t)\leq 0 ~\text{a.e. in} \, \mathcal{O}\}\subset{L^2(\mathcal{O})}.
$

By implementing the CG method to solve the $u$-subproblem (\ref{ADMM_us}) subject to the inexactness criterion (\ref{eu}), an ADMM--CG iterative scheme similar as Algorithm \ref{ADMM_CG} can be obtained for the problem (\ref{wave_control})--(\ref{wave_eqn}). For numerical discretization, we employ the central difference method (with step size $\tau$) for the time discretization and piecewise linear finite element method (with mesh size $h$) for the space discretization. All notations and remarks in Section \ref{sec:numerical} are used here again.
Let $\beta=5$ and $tol=10^{-5}$. We test the cases where the space mesh size $h$ and the time step size $\tau$ are $h=\tau={2^{-i}}$ with $i=5, 6, 7, 8$. Numerical results are presented in Table \ref{Table_supp}.
\begin{table}[ht]
	\setlength{\abovecaptionskip}{0pt}
	\setlength{\belowcaptionskip}{4pt}
	\centering
	\caption{Numerical comparison of ADMM--CG and $\hbox{ADMM}_{1e-k}$ for Example 3.}\label{Table_supp}
	{\footnotesize\begin{tabular}{|c|c|c|c|c|c|c|}
			\hline
			{ Mesh}&{ Algorithm}&${\tiny\text{ADMM}_{Iter}}$&{ Mean/Max CG}&{ Time (s)}&{ RelDis}&{ Obj}\\
			\hline
			
			&$\hbox{ADMM}_{1e-10}$ &46&17.69/72&16.26&3.8248$\times 10^{-3}$&1.6716$\times 10^{-3}$\\
			\cline{2-7}
			&$\hbox{ADMM}_{1e-8}$
			&46&12.75/31&11.89&3.8248$\times 10^{-3}$&1.6716$\times 10^{-3}$\\
			\cline{2-7}				
			$2^{-5}$&$\hbox{ADMM}_{1e-6}$&46&8.54/18&8.05&3.8248$\times 10^{-3}$&1.6716$\times 10^{-3}$\\
			\cline{2-7}	
			&$\hbox{ADMM}_{1e-4}$ &46&4.89/11&4.86&3.8248$\times 10^{-3}$&1.6716$\times 10^{-3}$\\
			\cline{2-7}
			&$\hbox{ADMM}_{1e-2}$ &$\sim$&$\sim$&$\sim$&$\sim$&$\sim$\\
			\cline{2-7}
			&{\bf{ADMM--CG}} &46&1.96/2&2.23&3.8248$\times 10^{-3}$&1.6716$\times 10^{-3}$\\
			\hline
			
			&$\hbox{ADMM}_{1e-10}$ &48&16.75/23&168.04&3.7670$\times 10^{-3}$&1.6197$\times 10^{-3}$\\
			\cline{2-7}	   &$\hbox{ADMM}_{1e-8}$&48&12.85/20&109.32&3.7670$\times 10^{-3}$&1.6197$\times 10^{-3}$\\
			\cline{2-7}	$2^{-6}$&$\hbox{ADMM}_{1e-6}$&48&8.77/15&89.06&3.7670$\times 10^{-3}$&1.6197$\times 10^{-3}$\\
			\cline{2-7}
			&$\hbox{ADMM}_{1e-4}$ &48&5.00/11&54.21&3.7670$\times 10^{-3}$&1.6197$\times 10^{-3}$\\
			\cline{2-7}
			&$\hbox{ADMM}_{1e-2}$ &$\sim$&$\sim$&$\sim$&$\sim$&$\sim$\\
			\cline{2-7}
			&{\bf{ADMM--CG}} &49&1.96/2&24.29&3.7670$\times 10^{-3}$&1.6197$\times 10^{-3}$\\
			
			\hline
			
			&$\hbox{ADMM}_{1e-10}$ &49&16.73/23&3511.81&3.7169$\times 10^{-3}$&1.5845$\times 10^{-3}$\\
			\cline{2-7}
			&$\hbox{ADMM}_{1e-8}$ &49&12.78/19&2198.52&3.7169$\times 10^{-3}$&1.5845$\times 10^{-3}$\\
			\cline{2-7}
			$2^{-7}$&$\hbox{ADMM}_{1e-6}$&49&8.76/15&1814.87&3.7169$\times 10^{-3}$&1.5845$\times 10^{-3}$\\
			\cline{2-7}
			&$\hbox{ADMM}_{1e-4}$ &50&4.90/11&1131.26&3.7169$\times 10^{-3}$&1.5845$\times 10^{-3}$\\
			\cline{2-7}
			&$\hbox{ADMM}_{1e-2}$&$\sim$&$\sim$&$\sim$&$\sim$&$\sim$\\
			\cline{2-7}
			&{\bf{ADMM--CG}} &50&1.96/2&415.58&3.7169$\times 10^{-3}$&1.5845$\times 10^{-3}$\\
			\hline
			
			&$\hbox{ADMM}_{1e-10}$ &50&16.42/22&49802.84&3.6863$\times 10^{-3}$&1.5643$\times 10^{-3}$\\
			\cline{2-7}
			&$\hbox{ADMM}_{1e-8}$ &50&12.46/19&31824.46&3.6863$\times 10^{-3}$&1.5643$\times 10^{-3}$\\
			\cline{2-7}			
			$2^{-8}$&$\hbox{ADMM}_{1e-6}$&50&8.54/15&24823.09&3.6863$\times 10^{-3}$&1.5643$\times 10^{-3}$\\
			\cline{2-7}
			&$\hbox{ADMM}_{1e-4}$ &50&4.94/11&10533.96&3.6863$\times 10^{-3}$&1.5643$\times 10^{-3}$\\
			\cline{2-7}
			&$\hbox{ADMM}_{1e-2}$ & $\sim$&$\sim$&$\sim$&$\sim$&$\sim$\\
			\cline{2-7}
			&{\bf{ADMM--CG}} &51&1.96/2&4561.64&3.6863$\times 10^{-3}$&1.5643$\times 10^{-3}$\\
			\hline
		\end{tabular}
	}
\end{table}

According to Table \ref{Table_supp}, the ADMM--CG iterative scheme is also very efficient for the general case of the problem (\ref{wave_control})--(\ref{wave_eqn}) where $\omega\subsetneq\Omega$ and $d=2$. Similar as the parabolic case, it suffices to solve the $u$-subproblem (\ref{ADMM_us}) inexactly subject to the criterion (\ref{inexact_criterion_s}). The independence of the convergence to the mesh size of discretization is also observed.

Evolutions of the residuals and objective functional values with respect to the outer ADMM iterations are plotted in Figure \ref{iterative_result_sup}. These curves indicate the fast convergence of the ADMM--CG, despite the fact that the theoretical worst-case convergence rate is only $O(1/K)$. In addition, the iterative errors $\|y^k-y^*\|$ and $\|u^k-u^*\|$ in Figure \ref{iterative_result_sup} (right) show that the discretization errors dominate the total errors of the numerical solution. This means the ADMM--CG finds a rather precise iterative solution very fast. The control variable $u$,  state variable $y$, and the errors $u^*-u$ and $y^*-y$ at $t=0.75$  with $h=\tau=2^{-6}$ are depicted in Figures \ref{numerical_result_sup} and \ref{error_sup}, respectively.

\begin{figure}[ht]
	\setlength{\abovecaptionskip}{0pt}
	\setlength{\belowcaptionskip}{5pt}
	\caption{Residuals (left), objective functional value (middle), and errors of $u$ and $y$ (right) with respect to the outer ADMM iterations for Example 3.}\label{iterative_result_sup}
	\centering
	\includegraphics[width=0.32\textwidth]{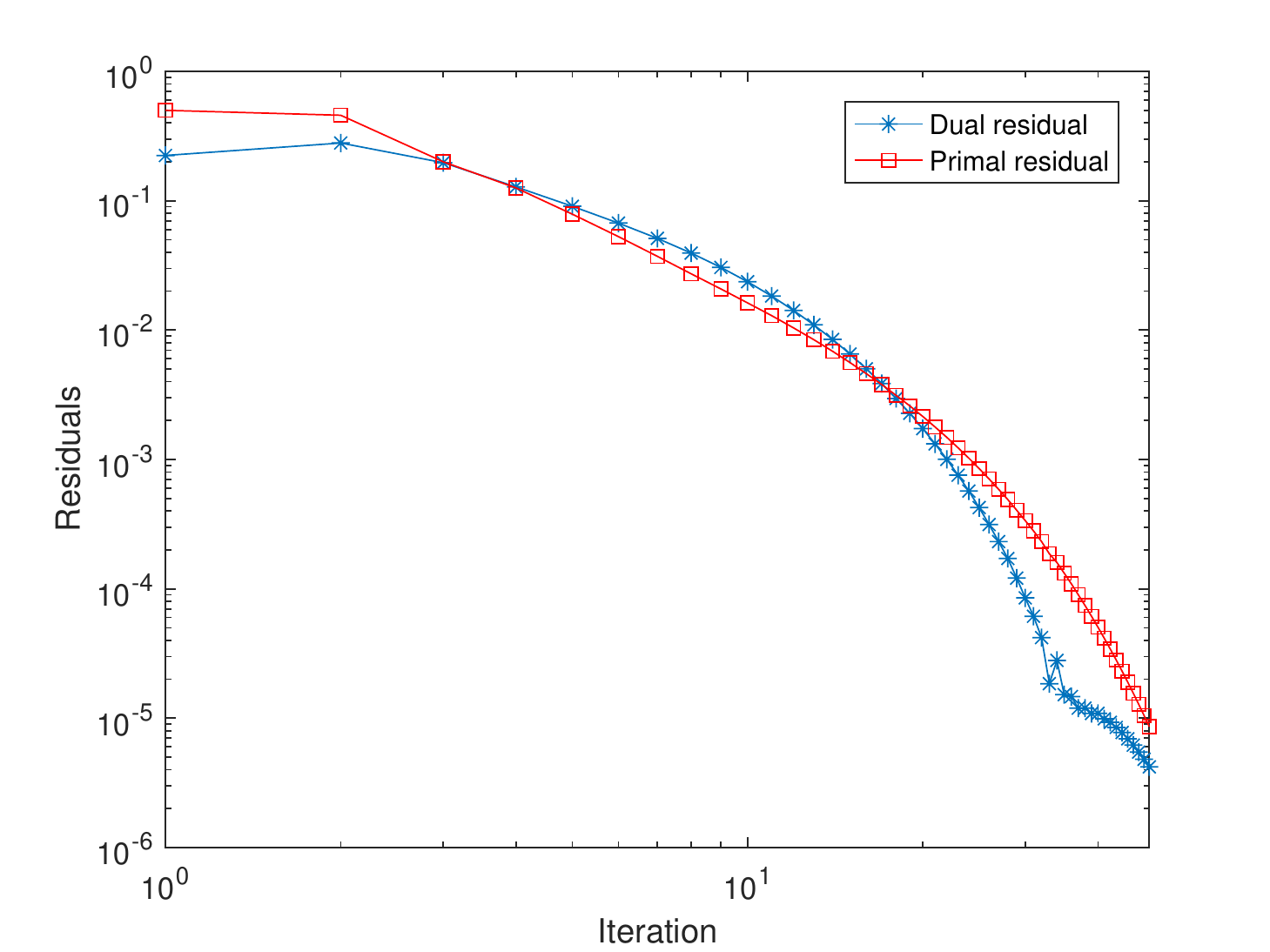}
	\includegraphics[width=0.32\textwidth]{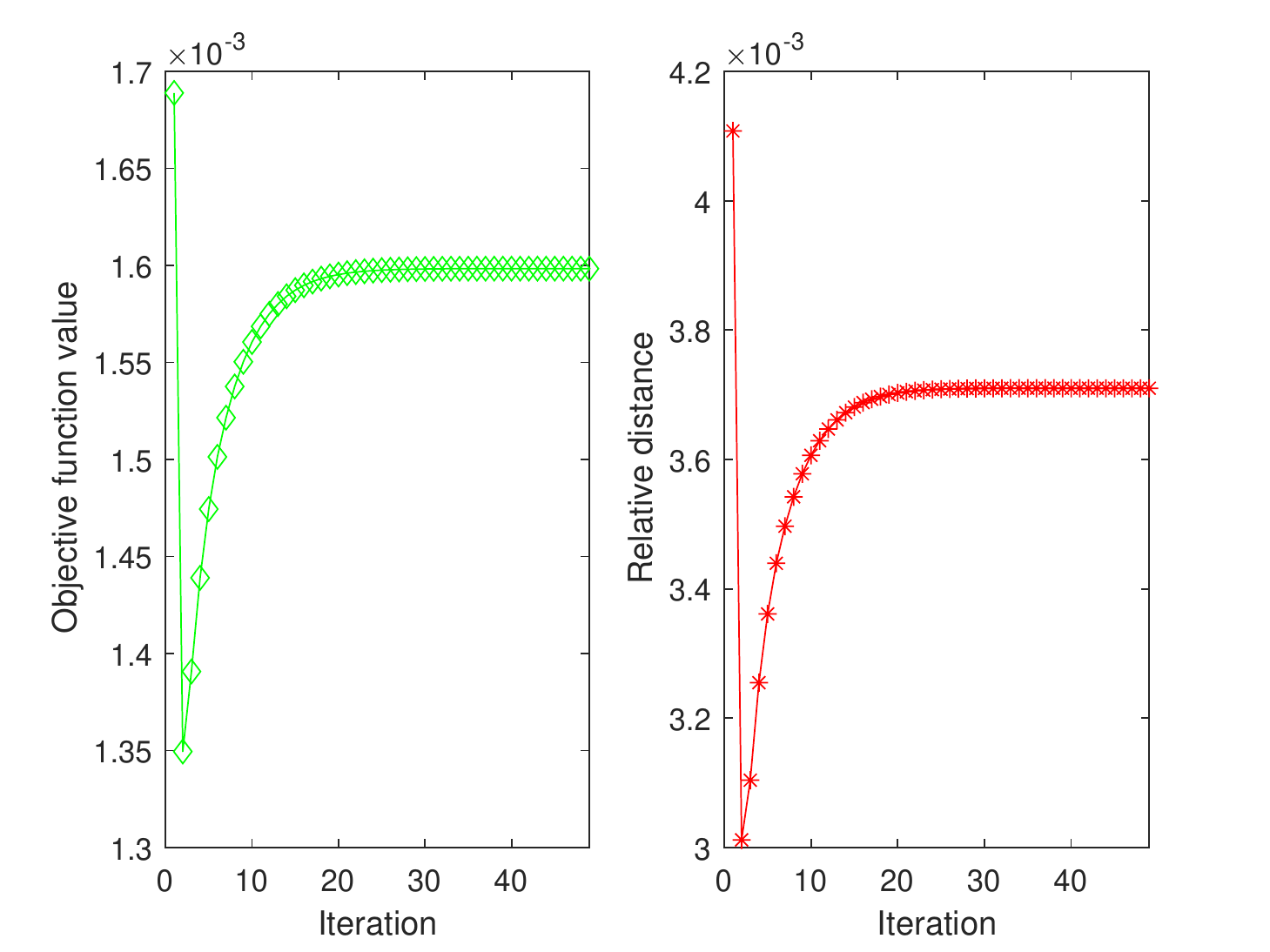}
	\includegraphics[width=0.32\textwidth]{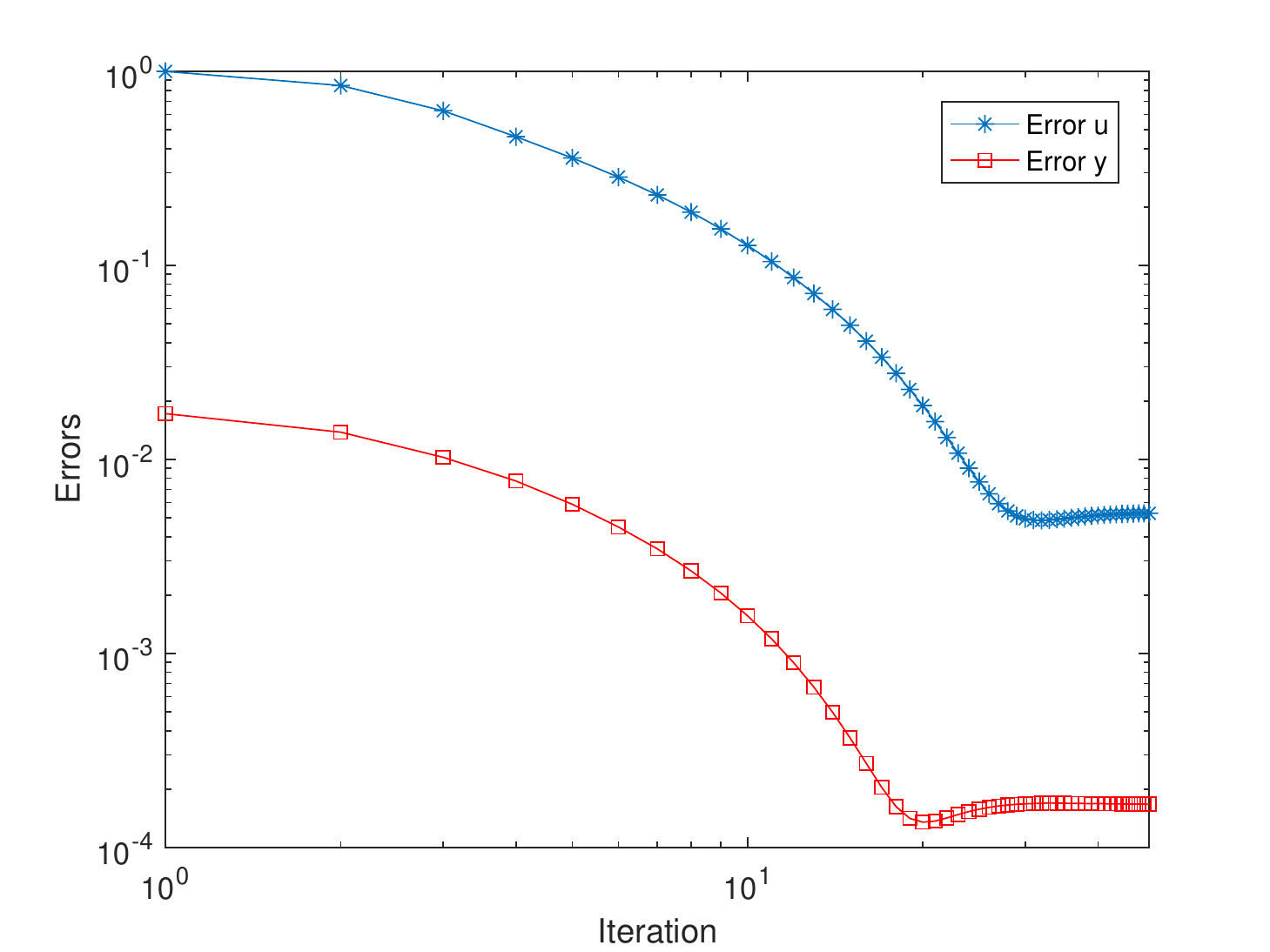}
\end{figure}

\begin{figure}[ht]\caption{Numerical solutions $u$ (left) and $y$ (right) at $t=0.75$ for Example 3.}\label{numerical_result_sup}
	\setlength{\abovecaptionskip}{0pt}
	\setlength{\belowcaptionskip}{5pt}
	\centering
	\includegraphics[width=0.48\textwidth]{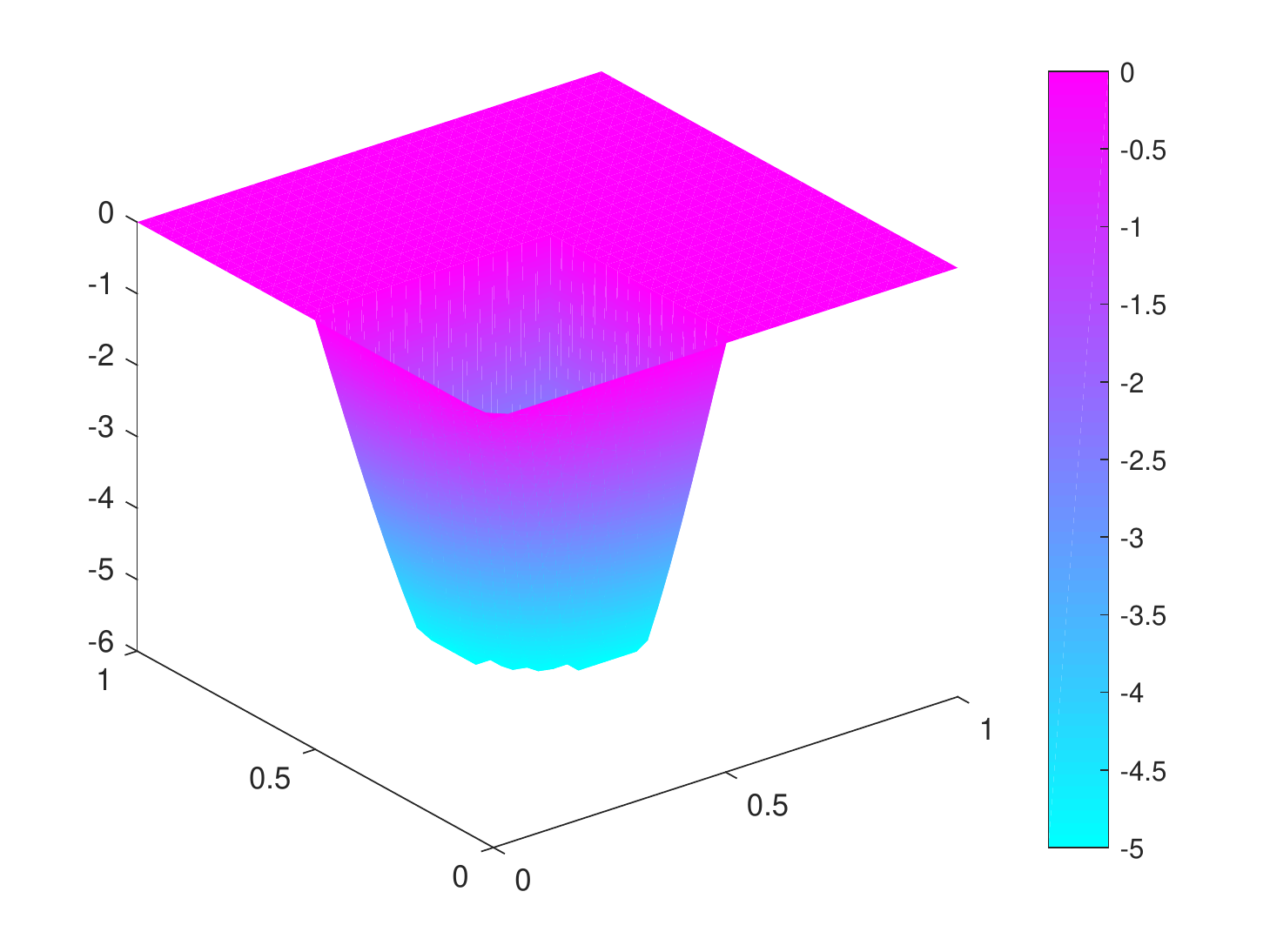}
	\includegraphics[width=0.48\textwidth]{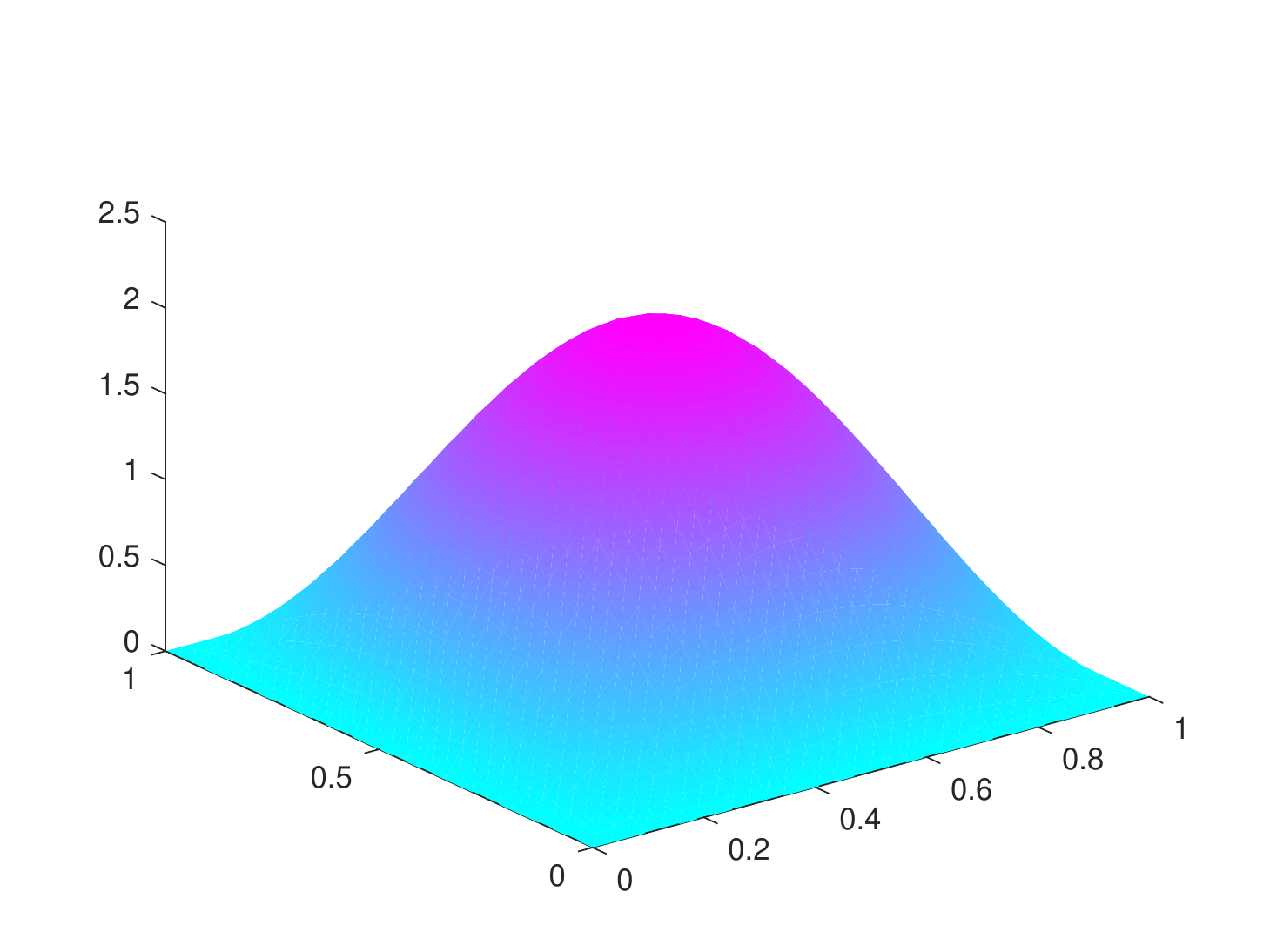}
\end{figure}

\begin{figure}[ht]
	\setlength{\abovecaptionskip}{0pt}
	\setlength{\belowcaptionskip}{5pt}
	\caption{Errors $u^*-u$ (left) and $y^*-y$ (right) at $t=0.75$ for Example 3.}\label{error_sup}
	\centering
	\includegraphics[width=0.48\textwidth]{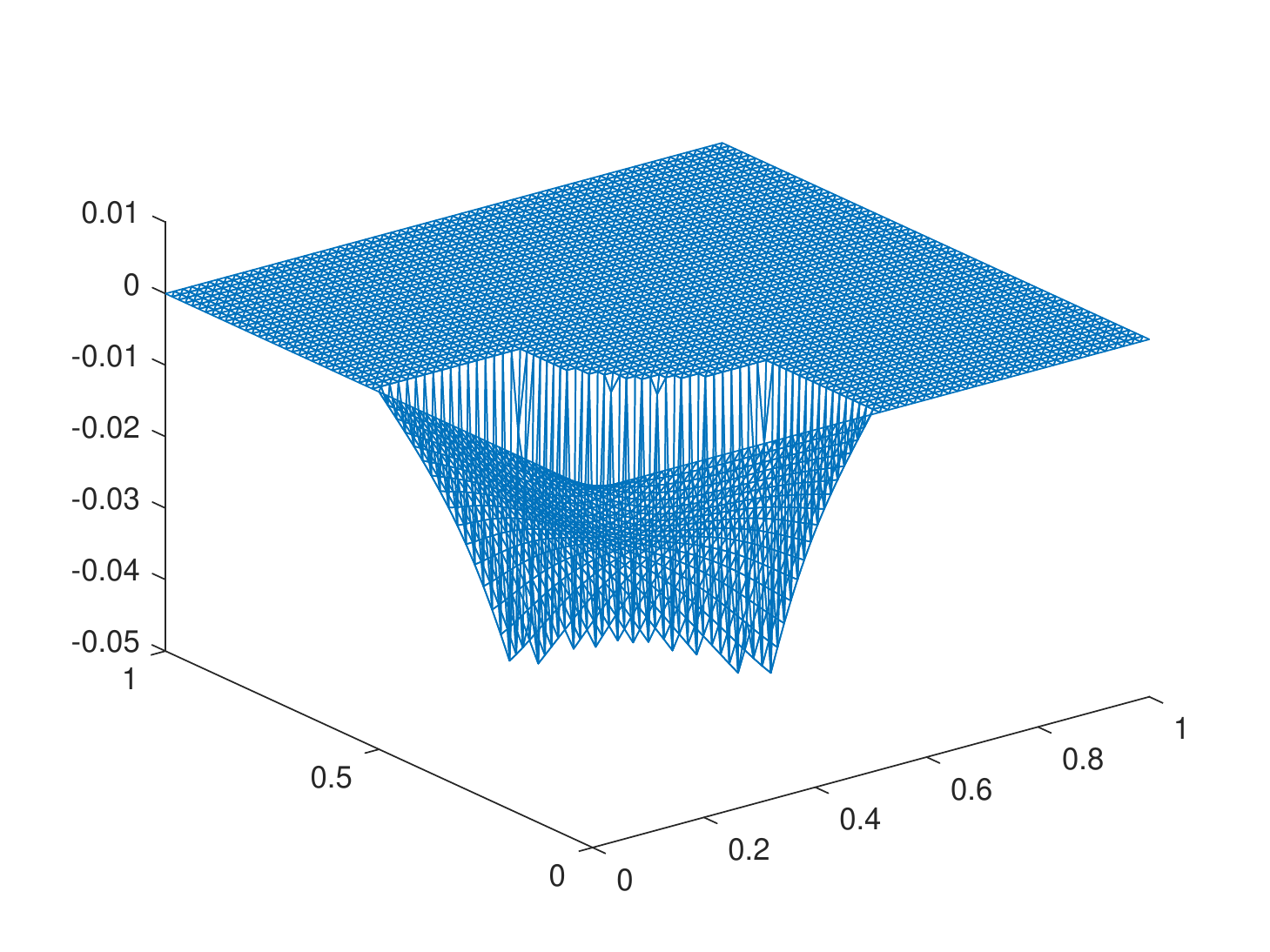}
	\includegraphics[width=0.48\textwidth]{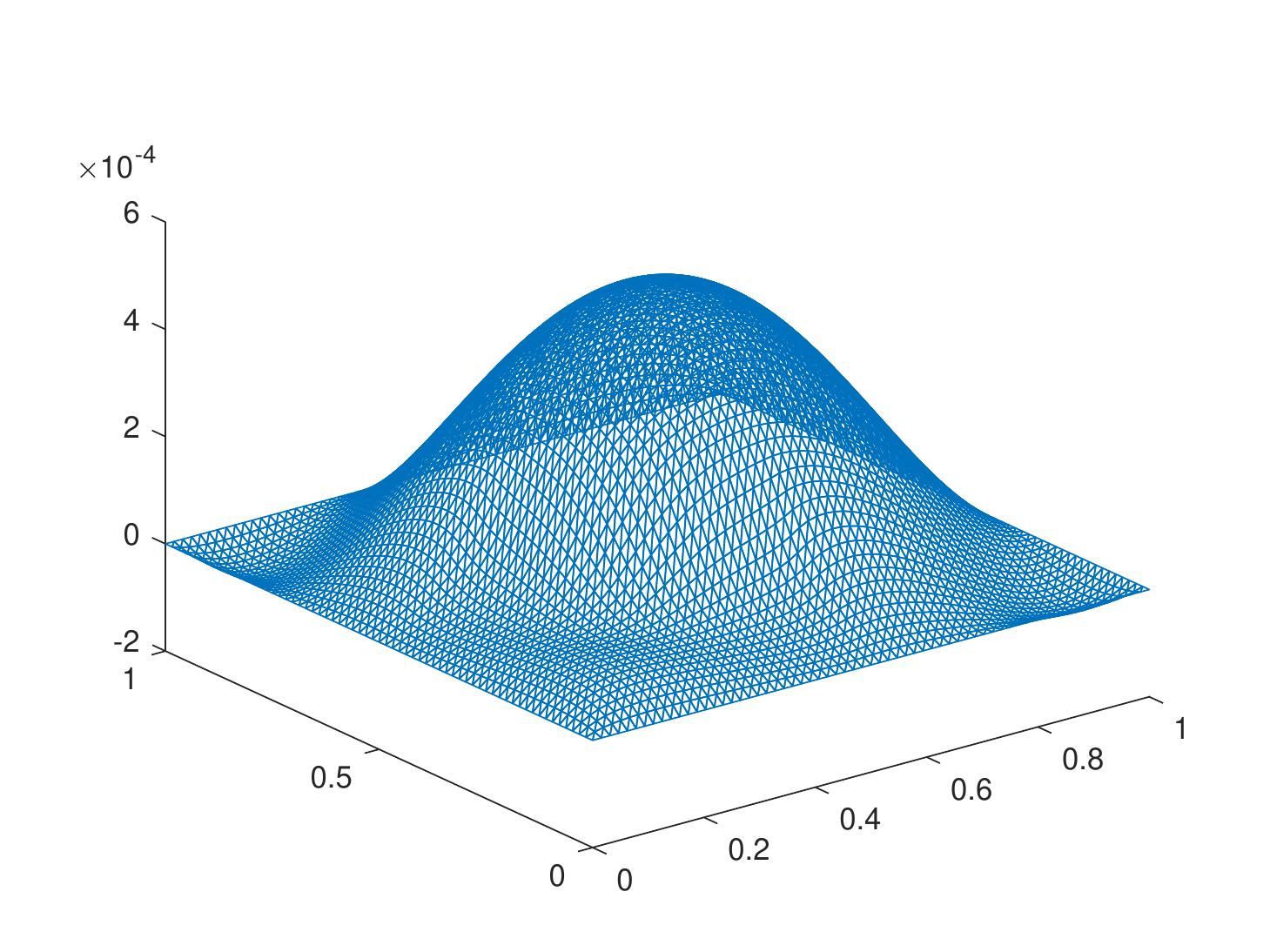}
\end{figure}

\subsection{Extension to Elliptic Optimal Control Problems with Control Constraints}

Our discussion can also be extended to various elliptic optimal control problems with control constraints. It is well known that SSN type methods are very efficient for solving elliptic optimal control problems, see, e.g., \cite{BKI99,hinze2008optimization,hinze12,KR2002,pearson2012,porcelli2015preconditioning,schiela2014operator,SW12} for a few references. In this subsection, we choose the particular SSN method in \cite{porcelli2015preconditioning} for numerical comparison.

\subsubsection{Model}

We consider the following elliptic optimal control problem with control constraints:
\begin{equation}\label{ell_obj}
\underset{u\in\mathcal{C},y\in H^1_0(\Omega)}{\min} J(y,u) =\frac{1}{2}\|{y}-{y_d}\|_{L^2(\Omega)}^2+\frac{\alpha}{2}\|{u}\|_{L^2(\Omega)}^2,
\end{equation}
subject to the
following elliptic equation:
\begin{equation}\label{ell_state}
-\Delta y=u ~ \text{in}~ \Omega,\quad
y=0~ \text{on}~ \Gamma,
\end{equation}
and the admissible set $\mathcal{C} $ is defined by
\begin{equation*}\label{admissible_set}
\mathcal{C} =\{u\in L^\infty(\Omega)| a\leq u(x_1,x_2)\leq b, ~\text{a.e.~in}~ \Omega \}\subset L^2(\Omega),
\end{equation*}
where $a$ and $b$ are given constants.

\subsubsection{Algorithm}

Similar as (\ref{ADMM}), implementation of the ADMM to the problem (\ref{ell_obj})--(\ref{ell_state}) is
\begin{subequations}\label{ADMM_se}
	\begin{numcases}
	~u^{k+1}=\arg\min_{u\in L^2(\Omega)}\tilde{L}_\beta(u, z^{k},\lambda^k),\label{ADMM_use}\\
	z^{k+1} =\arg\min_{z\in L^2(\Omega)}\tilde{L}_\beta(u^{k+1},z,\lambda^k),\label{ADMM_zse}\\
	\lambda^{k+1} = \lambda^k-\beta(u^{k+1}-z^{k+1}).\label{ADMM_lambdase}
	\end{numcases}
\end{subequations}
Above, the augmented Lagrangian functional $\tilde{L}_\beta(u, z,\lambda)$ is defined as
\begin{equation*}
\tilde{L}_\beta(u, z,\lambda):=\tilde{J}(u)+I_{\mathcal{C}}(z)-(\lambda,u-z)_{L^2(\Omega)}+\frac{\beta}{2}\|u-z\|_{L^2({\Omega})}^2,
\end{equation*}
where
$\tilde{J}(u) =\frac{1}{2\alpha}\|Su-{y_d}\|_{L^2(\Omega)}^2+\frac{1}{2}\|{u}\|_{L^2(\Omega)}^2$,
and $S: L^2(\Omega)\longrightarrow L^2(\Omega)$ is the solution operator associated with the elliptic equation (\ref{ell_state}).

Similarly, it is easy to show that the $z$-subproblem (\ref{ADMM_zse}) is essentially computing the projection onto the admissible set $\mathcal{C}$; and the $u$-subproblem (\ref{ADMM_use}) is an unconstrained optimal control problem subject to the elliptic equation (\ref{ell_state}), which can be iteratively solved by some existing methods, e.g., the preconditioned MinRes method in \cite{pearson2012}. In a way similar as what we have done for the problem (\ref{Basic_Problem})--(\ref{state_eqn}), we can propose the following inexactness criterion for solving the $u$-subproblem (\ref{ADMM_use}) inexactly:
\begin{equation}\label{ell_inexact}
\|e_k(u^{k+1})\|_{L^2(\Omega)}\leq\sigma \|e_k(u^k)\|_{L^2(\Omega)},
\end{equation}
where the constant $\sigma$ is given in (\ref{sigma}) and $e_k(u)$ is defined as
\begin{equation*}
e_k(u)=(1+\beta)u+p-\beta z^k-\lambda^k.
\end{equation*}
Here, the adjoint variable $p$ is the solution of the following adjoint equation:
$$-\Delta p=\frac{1}{\alpha}(y-y_d)~\text{in}~\Omega,\quad p=0~ \text{on}~\Gamma.$$
Hence, an inexact version of the ADMM (\ref{ADMM_se}) can be proposed for the problem (\ref{ell_obj})--(\ref{ell_state}) by changing the inexactness criterion (\ref{inexact_criterion}) in Algorithm \ref{InADMM_algorithm} as the one defined in (\ref{ell_inexact}). For succinctness, we omit the details.

\subsubsection{Numerical Results}

Now, we test the ADMM (\ref{ADMM_se}) with the inexactness criterion (\ref{ell_inexact}) for the problem (\ref{ell_obj})--(\ref{ell_state}), and compare it with the SSN method in \cite{porcelli2015preconditioning}.

\medskip
\noindent\textbf{Example 4}. Let $\Omega=\{(x_1,x_2)\in \mathbb{R}^2|0<x_1<1, 0<x_2<1\}$. We consider the example given in \cite{hinze12}, where the admissible set is specified as
$$
\mathcal{C} =\{u\in L^\infty(\Omega)| 0.3\leq u(x_1,x_2)\leq 1, ~\text{a.e.~in}~ \Omega \}\subset L^2(\Omega),
$$
and the desired state is given by
$
y_d = 4\pi^2\alpha \sin (\pi x_1) \sin (\pi x_2)+	y_r.
$
Here, the function $y_r$ denotes the solution to the problem
\begin{equation*}
-\Delta y_r=r ~\text{in}~\Omega,\quad
y_r=0~\text{on}~ \Gamma,
\end{equation*}
where $r = \min\left \{1, \max\left \{0.3, 2\sin(\pi x_1)\sin(\pi x_2)\right \}\right\}$. It follows from the construction of $y_d$ and $r$ that $u^* := r$ is the unique solution of this example.

To solve the resulting $u$-subproblem (\ref{ADMM_use}) and meet the inexactness criterion (\ref{ell_inexact}), we first derive its dual problem which is an unconstrained quadratic optimization problem in terms of the adjoint variable $p$, and then employ a preconditioned conjugate gradient (PCG) method (see e.g., Algorithm 2.3 in \cite{schiela2014operator}) with the preconditioner proposed in \cite{pearson2012}. Accordingly, an ADMM--PCG iterative scheme can be proposed for the problem (\ref{ell_obj})--(\ref{ell_state}). To implement it, we set the initial values as $u=0.5$, $z=0$, $\lambda=0$, the penalty parameter $\beta=2$, and tolerance $tol=10^{-7}$.

For the numerical implementation of the SSN method in \cite{porcelli2015preconditioning}, we follow all steps in the original paper, including the finite element discretization, the preconditioned GMRES solver for Newton systems, and the stopping criteria for inner iterations. The initial values of the SSN method are set as $u=0.5$, $y=0.5$, $p=0$ and $\mu=0$, where $\mu=\mu_a+\mu_b$ with $\mu_a,\mu_b$ the Lagrange multipliers associated with the lower and upper bound of control constraints, as defined by (2.2) in \cite{porcelli2015preconditioning}. We follow \cite{porcelli2015preconditioning}
and terminate the SSN iterations when the nonlinear residual $F(u_k; y_k; p_k; \mu_k)$ (see (2.4) in \cite{porcelli2015preconditioning}) is sufficiently small, i.e.,
$F(u_k; y_k; p_k; \mu_k)\leq 10^{-8}$. We set $\alpha=10^{-4}$ in (\ref{ell_obj}) and test various mesh sizes $h=2^{-i}$ with $i=5,6,7,8,9$. Numerical results of the SSN in \cite{porcelli2015preconditioning} and the ADMM--PCG iterative scheme are reported in Table \ref{com_table}.
\begin{table}[ht]
	\setlength{\abovecaptionskip}{0pt}
	\setlength{\belowcaptionskip}{5pt}
	\centering
	\caption{Numerical comparison of the SSN in \cite{porcelli2015preconditioning} and the ADMM--PCG for Example 4.}
	{\small\begin{tabular}{|c|c|c|c|c|c|}
			\hline  Algorithm&$h$&No. of outer iterations & CPU Time (s)&$\|u-u^*\|_{L^2(\Omega)}$ \\
			\hline    &$2^{-5}$& 5 & 0.4817   & $5.8269\times10^{-5}$
			\\
			\cline{2-5}
			
			&$2^{-6}$&6 &  0.8948  & $1.4671\times10^{-5}$
			\\
			\cline{2-5}
			SSN &$2^{-7}$& 6&  3.8564  & $3.6631\times10^{-6}$
			\\
			\cline{2-5}
			&$2^{-8}$& 6& 13.6203   & $9.1543\times10^{-7}$
			\\
			\cline{2-5} &$2^{-9}$& 6& 54.7350   & $2.2885\times10^{-7}$
			\\
			\hline    &$2^{-5}$&41 &  0.3211   & $5.8405\times10^{-5}$
			\\
			\cline{2-5}
			
			&$2^{-6}$&43& 0.6071   & $1.4676\times10^{-5}$
			\\
			\cline{2-5}
			ADMM--PCG &$2^{-7}$&42 &  2.1962  & $3.7458\times10^{-6}$
			\\
			\cline{2-5}
			&$2^{-8}$& 41&   8.1225 & $9.2369\times10^{-7}$
			\\
			\cline{2-5} &$2^{-9}$&41 & 32.5952   & $2.3482\times10^{-7}$
			\\
			\hline
		\end{tabular}\label{com_table}}
\end{table}

From Table \ref{com_table}, we observe that the ADMM--PCG converges even faster than the SSN method in \cite{porcelli2015preconditioning}, especially when mesh sizes are small. It requires more iteration numbers, but its computation load per iteration is much less because it is free from solving Newton systems in its iterations. Hence, the ADMM-PCG is another efficient method that can be used for elliptic optimal control problems.

\section{Conclusions}\label{sec:conclusion}
In this paper, we focused on the implementation of the well-known alternating direction method of multipliers (ADMM) to parabolic optimal control problems with control constraints. Direct implementation of ADMM decouples the control constraint and the parabolic state equation at each iteration, while the resulting unconstrained parabolic optimal control subproblems should be solved inexactly. Hence, only inexact versions of the ADMM are implementable for these problems. We proposed an easily implementable inexactness criterion for these subproblems; and obtained an inexact version of the ADMM whose execution consists of two-layer nested iterations. The strong global convergence of the resulting inexact ADMM was proved rigorously in an infinite-dimensional Hilbert space; and the worst-case convergence rate measured by the iteration complexity was also established. We illustrated by the CG method how to execute the inexactness criterion, and showed the efficiency of the resulting ADMM--CG iterative scheme numerically. In particular, our numerical results validate that usually a few internal CG iterations are sufficient to guarantee the overall convergence of the ADMM--CG; hence there is no need to solve the unconstrained parabolic optimal control problem at each iteration up to a high precision. This fact significantly saves computation and contributes to the efficiency of the ADMM--CG. As mentioned in Remark \ref{remark-criterion}, the new inexactness criterion possesses a variety of features that are software-friendly and hence important for softwarization and industrialization. In this sense, we follow the fundamental concept of trustworthiness in software engineering (also in artificial intelligence) and call the proposed inexact ADMM, or more concretely Algorithm \ref{ADMM_CG}, a trustworthy algorithm.

We also briefly discussed how to extend the inexact ADMM to other optimal control problems, including optimal control problems constrained by the wave equation with control constraints, and elliptic optimal control problems with control constraints. Our philosophy in algorithmic design can be easily extended to these problems; hence the proposed inexact ADMM can be deliberately specified as various algorithms for a wide range of optimal control problems. For some challenging problems whose numerical study is limited (such as the general case of (\ref{Basic_Problem})--(\ref{state_eqn}) or (\ref{wave_control})--(\ref{wave_eqn}) where $\omega\subsetneq \Omega$ and $d\ge2$), the algorithms specified from the inexact ADMM are attractive in senses of numerical performance and easiness of coding. For some relatively easier problems that have been well studied (such as elliptic optimal control problems), the algorithms specified from the inexact ADMM can also be very competitive with state-of-the-art iterative schemes in the literature. It is interesting and much more challenging to design operator splitting type algorithms for optimal control problems constrained by some nonlinear PDEs in the future.

\bibliographystyle{amsplain}
\end{document}